\documentclass[11pt]{article}

\RequirePackage[colorlinks]{hyperref}

\usepackage{amsmath}
\usepackage{amsthm}
\usepackage{amssymb}
\usepackage{verbatim}
\usepackage{url}
\usepackage{fullpage}
\usepackage{graphicx}

\allowdisplaybreaks

\newtheorem{theorem}{Theorem}[section]
\newtheorem{proposition}[theorem]{Proposition}
\newtheorem{cor}[theorem]{Corollary}
\newtheorem{lemma}[theorem]{Lemma}
\newtheorem{definition}[theorem]{Definition}

\begin{document}

\title{}
\author{\textbf{A Dirichlet Process Characterization of RBM in a Wedge} \\
~\\
Peter Lakner, Josh Reed and Bert Zwart}
\maketitle{}

\begin{abstract}
Reflected Brownian motion (RBM) in a wedge is a 2-dimensional stochastic  process $Z$ whose state space in $\mathbb{R}^2$ is given in polar coordinates by $S=\{(r,\theta): r \geq 0, 0 \leq \theta \leq \xi\}$ for some $0 < \xi < 2 \pi$. Let $\alpha= (\theta_1+\theta_2)/\xi$, where $-\pi/2 < \theta_1,\theta_2 < \pi/2$ are the directions of reflection of $Z$ off each of the two edges of the wedge as measured from the corresponding inward facing normal. We prove that in the case of $1 < \alpha < 2$, RBM in a wedge is a Dirichlet process. Specifically, its unique Doob-Meyer type decomposition is given by $Z=X+Y$, where $X$ is a two-dimensional Brownian motion and $Y$ is a continuous process of zero energy. Furthermore, we show that for $ p > \alpha $, the strong $p$-variation of the sample paths of $Y$ is finite on compact intervals, and, for $0 < p \leq \alpha $,  the strong $p$-variation of $Y$ is infinite on $[0,T]$ whenever $Z$ has been started from the origin. We also show that on excursion intervals of $Z$ away from the origin, $(Z,Y)$ satisfies the standard Skorokhod problem for $X$. However, on the entire time horizon $(Z,Y)$ does not satisfy the standard Skorokhod problem for $X$, but nevertheless we show that it satisfies the extended Skorkohod problem.
\end{abstract}

\section{Introduction and Main Results}
\label{Section:Introduction}

Reflected Brownian motion (RBM) in a wedge, as defined by Varadhan and Williams \cite{varadhan1985brownian}, is a 2-dimensional stochastic process $Z$ whose state space in $\mathbb{R}^2$ is given in polar coordinates by $S=\{(r,\theta): r \geq 0, 0 \leq \theta \leq \xi\}$ for some $0 < \xi < 2 \pi$. In the interior of $S$, RBM in a wedge behaves as a standard 2-d Brownian, whereas upon hitting the boundary of $S$, it is reflected back into the interior at an angle that depends upon which edge of the boundary has been hit. Let $\partial S_1=\{(r,\theta): r \ge 0, \theta = 0\}$ and $\partial S_2=\{(r,\theta): r \ge 0, \theta = \xi\}$. The angle of reflection off of $\partial S_1\setminus\{0\}$ is denoted by $ \theta_1$, and the angle of reflection off of  $\partial S_2\setminus\{0\}$ is denoted by $ \theta_2 $. Both of these angles are measured with respect to the inward facing normal off of each edge, with angles directed toward the origin assumed to be positive. See Figure \ref{Figure:State:Space:Of:RBM:Wedge} below for an illustration of the above quantities. The angle of reflection at the vertex of the wedge is discussed later in the paper.

In \cite{varadhan1985brownian}, Varadhan and Williams provided the following rigorous definition of RBM in a wedge as the solution to  a submartingale problem. For each $j=1,2,$ denote the direction of reflection off of the edge $\partial S_j\setminus\{0\}$ by $v_j$. Assume moreover that $v_j$ is normalized such that $v_j \cdot n_j =1$, where $n_j$ is the inward facing unit normal vector off of the edge $\partial S_j\setminus\{0\}$. Let $C_S$ denote the space of continuous functions with domain $\mathbb{R}_{+}=[0,\infty)$ and range $S$. For each $t \geq 0$ and $\omega \in C_S$, we denote by $Z(t):C_S \mapsto S$ the coordinate map $Z(t)(\omega)=Z(t,\omega)=\omega(t)$, and we also define the coordinate mapping process $Z=\{Z(t), t \geq 0\}$. Then, for each $t \geq 0$,
we set $\mathcal{M}_t = \sigma\{Z(s), 0 \leq s \leq t\}$, and set $\mathcal{M} = \sigma\{Z(s),  s \geq 0\}$. Next, for each $n \geq 1$ and $F \subset \mathbb{R}^2$, denote by $C^n(F)$ the set of $n$-times continuously differentiable functions in some domain containing $F$ and let $C^n_b$ be the set of functions in $C^n(F)$ that have bounded partial derivatives up to and including order $n$ on $F$. Finally, define the differential operators $D_j = v_j \cdot \nabla$ for $j=1,2,$ and denote by $\triangle$ the Laplacian operator.

\begin{figure}[h]
\begin{center}
\scalebox{.75}{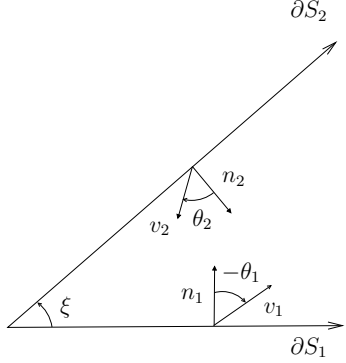~~~~~~~~}
\caption{The state space $S$ of RBM in a wedge.}
\label{Figure:State:Space:Of:RBM:Wedge}
\end{center}
\end{figure}

\begin{definition}\label{vw}[Varadhan and Williams \cite{varadhan1985brownian}] A family of probability measures $\{P^z, z \in S\}$ on $(C_S,\mathcal{M})$ is said to solve
the submartingale problem if for each $z \in S$, the following 3 conditions hold,
\begin{enumerate}
\item~$P^z(Z(0)=z)~=~1$, \label{Condition:1:Submartingale}

\item For each $f \in C^2_b(S)$, the process
\begin{eqnarray*}
\left\{ f(Z(t)) - \frac{1}{2} \int_0^t \triangle f(Z(s))ds,~~t \geq 0 \right\}
\end{eqnarray*}
is a submartingale on $(C_S,\mathcal{M},\mathcal{M}_t,P^z)$ whenever $f$ is constant in a neighborhood of the origin and satisfies $D_i f \geq 0$ on $\partial S_i$ for $i=1,2,$  \label{Condition:2:Submartingale}

\item
$$E^z \left[ \int_0^{\infty} 1\{Z(t)=0\}dt\right] ~=~0.
$$  \label{Condition:3:Submartingale}
\end{enumerate}
\end{definition}

Now define
\begin{eqnarray*}
\alpha &=& \frac{\theta_1+\theta_2}{\xi}.
\end{eqnarray*}

In \cite{varadhan1985brownian}, it was shown that if $\alpha < 2$, then there exists a unique solution $\{P^z, z \in S\}$ to the submartingale problem. Moreover \cite{varadhan1985brownian}, the family $\{P^z, z \in S\}$ possess the strong Markov property. If $\alpha \geq 2$, then there is no solution \cite{varadhan1985brownian} to the submartingale problem but there is a unique family of probability measures $\{P^z, z \in S\}$ satisfying Conditions \ref{Condition:1:Submartingale} and \ref{Condition:2:Submartingale} of Definition \ref{vw}. In this case, we may continue to refer to $\{P^z, z \in S\}$ as the solution to the submartingale problem, with the understanding that only Conditions \ref{Condition:1:Submartingale} and \ref{Condition:2:Submartingale} of Definition \ref{vw} are satisfied.

Now, as in Williams \cite{williams1985recurrence}, let $\mathcal{B}_S$ denote the Borel $\sigma$-algebra on $S$ and for each probability measure $\mu$ on $(S,\mathcal{B}_S)$, define the measure $P^{\mu}$ on $(C_S,\mathcal{M})$ by setting
\begin{eqnarray*}
P^{\mu}(A)&=& \int_S \mu(dz)P^z(A),~A \in \mathcal{M}.
\end{eqnarray*}
We then denote by $(C_S,\mathcal{M}^{\mu},P^{\mu})$ the completion of $(C_S,\mathcal{M},P^{\mu})$
and for each $t \geq 0$, we let $\mathcal{M}^{\mu}_t$ be the augmentation of $\mathcal{M}_t$ with respect to $(C_S,\mathcal{M}^{\mu},P^{\mu})$. If $\mu=\delta_{\{z\}}$ for some $z\in S$ (the Dirac measure centered at $z$), then we shall write ${\mathcal M}^z$ and  ${\mathcal M}_t^z$ instead of ${\mathcal M}^\mu$ and  ${\mathcal M}_t^\mu$.
Next, define  $\mathcal{F} = \bigcap_{\mu}{\mathcal M}^{\mu}$, where the intersection is taken over all $\mu$ which are probability measures on $(S,\mathcal{B}_S)$, and similarly set  $\mathcal{F}_t = \bigcap_{\mu}M^{\mu}_t$. It then follows \cite{williams1985recurrence} that the filtration $\mathcal{F}_t$ is right-continuous and, moreover, the process $\mathcal{Z}= (C_S,\mathcal{F},\mathcal{F}_t,Z(t),\theta_t,P^z) $ is a Hunt process with state space $(S,\mathcal{B}_S)$, where $\theta_t$ is the usual shift operator.

\subsection{Dirichlet Process Result}

In \cite{williams1985reflected}, Williams proved that for each $z \in S$, $Z$ is a semi-martingale on $(C_S,\mathcal{F},\mathcal{F}_t,P^z)$ if and only if $\alpha < 1$ or $\alpha \geq 2$ . More recently, Kang and Ramanan \cite{kang2010dirichlet} have proven  that in the case of $\alpha=1$, $Z$ is a Dirichlet process (see Definition \ref{Definition:Dirichlet:Process} below) on $(C_S,\mathcal{F},\mathcal{F}_t,P^z)$ for each $z \in S$. In Theorem \ref{Theorem:Main:Dirichlet:Process:Result} below, we complement the results of Williams \cite{williams1985reflected} and Kang and Ramanan \cite{kang2010dirichlet} by proving that in the case of $1 < \alpha < 2$, $Z$ is a Dirichlet process on $(C_S,\mathcal{F},\mathcal{F}_t,P^z)$ for each $z \in S$.

For each $E \subseteq \mathbb{R}_{+}$, let
\begin{eqnarray*}
\pi(E) &=&\{\pi=\{t_i,i=0,...,n\} \subseteq E : n \in \mathbb{N}_0, ~t_{i-1} \leq t_i~,i=1,...,n\}
\end{eqnarray*}
denote the set of all partitions of $E$, and define the mesh of a partition $\pi \in \pi(E)$ by setting
\begin{eqnarray*}
\| \pi \| &=& \max\{|t_i-t_{i-1}|:i=1,...,n\}.
\end{eqnarray*}
We then recall the definition of a continuous process with zero energy (see, for instance, Definition 3.2 of Kang and Ramanan \cite{kang2010dirichlet}).

\begin{definition}\label{Definition:Zero:Energy}
Let $z \in S$. A $2$-dimensional continuous process  $Y$ defined on $(C_S,\mathcal{F},\mathcal{F}_t,P^z) $ is said to be of zero energy if for each $T > 0$,
\begin{eqnarray*}
\sum_{t_i \in \pi^n}\|Y(t_i)-Y(t_{i-1})\|^2 &{\buildrel P \over \rightarrow} &0~\textrm{as}~n \rightarrow \infty,
\end{eqnarray*}
for any sequence $\{\pi^n, n \geq 1\}$ of partitions of $[0,T]$ with $\| \pi^n \| \rightarrow 0$ as $n \rightarrow \infty$.
\end{definition}

\noindent We now define a Dirichlet process as follows. (see Definition 2.4 of Coquet et al. \cite{coquet2006natural} or Definition 3.3 of \cite{kang2010dirichlet}).
\begin{definition}\label{Definition:Dirichlet:Process}
Let $z \in S$. The stochastic process $Z$ is said to be a Dirichlet process on $(C_S,\mathcal{F},\mathcal{F}_t,P^z)$ if we may write
\begin{eqnarray}
Z &=& X + Y,  \label{Display:Decomposition:Of:Z}
\end{eqnarray}
where $X$ is an $\mathcal{F}_t$-adapted local martingale and $Y$ is a continuous, $\mathcal{F}_t$-adapted zero energy process with $Y(0)=0$.
\end{definition}
\noindent The above definition of a Dirichlet process differs slightly from that of F\"{o}llmer \cite{follmer1981dirichlet}. However, it is sufficiently strong in order to enable one to obtain
formulas for performing certain operations on $Z$, such as a change-of-variables (see F\"{o}llmer \cite{follmer1981calcul}). Note also that if $Z$ is a semi-martingale, then it is also a Dirichlet process as defined above. However, the converse is not necessarily true. It may also be shown that the Doob-Meyer type decomposition of $Z$ given by \eqref{Display:Decomposition:Of:Z} is unique (see Remark 3.4 of \cite{kang2010dirichlet}).

\begin{theorem}\label{Theorem:Main:Dirichlet:Process:Result} Suppose that $1 < \alpha < 2$. Then, $Z$ has the decomposition
\begin{eqnarray}
Z&=&X+Y,\label{decomposition}
\end{eqnarray}
where  $(X,Y)$ is a pair of processes on $(C_S,\mathcal{F},\mathcal{F}_t)$
such that for each $z \in S$, $X$ is a standard Brownian motion started from $z$ and $Y$ is a process of zero energy on $(C_S,\mathcal{F},\mathcal{F}_t, P^z)$. In particular, for each $z\in S$, the process $Z$ is a Dirichlet process on $(C_S,\mathcal{F},\mathcal{F}_t,P^z)$.
\end{theorem}
\noindent
\noindent We emphasize that the pair of processes $(X,Y)$ appearing in Theorem \ref{Theorem:Main:Dirichlet:Process:Result} do not depend on $z\in S$. That is, there is a single pair of processes $(X,Y)$ on $(C_S,\mathcal{F},\mathcal{F}_t)$ such that the statement of the above theorem hold for each $z\in S$.

\subsection{Roughness of the Paths of $Y$ Result}

Theorem \ref{Theorem:Main:Dirichlet:Process:Result} implies that in the case of $1  < \alpha < 2$, the paths of $Y$ are of zero energy on $(C_S,\mathcal{F},\mathcal{F}_t,P^z)$ for each $z \in S$. We now turn our attention to providing a refinement of this result using the concept of strong $p$-variation.

\begin{definition}\label{Definition:Strong:P:Variation}
Let $E \subset \mathbb{R}_{+}$ and $p > 0$. The strong $p$-variation of a function $f: \mathbb{R}_{+} \mapsto \mathbb{R}^n$ on $E$ is defined by
\begin{eqnarray*}
V_p(f,E) &=&  \sup \left\{ \sum_{t_i \in \pi}\|f(t_i)-f(t_{i-1})\|^p  : \pi \in \pi(E)\right\}.
\end{eqnarray*}
\end{definition}

\begin{theorem}\label{Theorem:Converse:Dirichlet:Process:Result}Suppose that  $1 < \alpha < 2$. Then, for each $p >  \alpha $ and $z \in S$,
\begin{eqnarray}
P^z (V_p(Y,[0,T]) < + \infty) &=&1,~T \geq 0. \label{Display:Condition:Main:Dirichlet}
\end{eqnarray}
Furthermore, for each $0 < p \leq \alpha$,
\begin{eqnarray}
P^0 (V_p(Y,[0,T]) < + \infty) &=&0,~T \geq 0. \label{Display:Condition:Converse:Main:Dirichlet}
\end{eqnarray}
\end{theorem}
\noindent

\noindent Since $ \alpha < 2$ in the statement of Theorem \ref{Theorem:Converse:Dirichlet:Process:Result}, it is straightforward to show using the relationship between functions of finite strong $p$-variation and $1/p$-H\"{o}lder continuous functions \cite{chistyakov1998maps} that \eqref{Display:Condition:Main:Dirichlet} automatically implies that $Y$ is of zero energy on $(C_S,\mathcal{F},\mathcal{F}_t,P^z)$ for each $z \in S$.

\subsection{Extended Skorokhod Map Result}
\label{Section:Extended:Skorokhod:Map:Results}

We now provide an additional connection between the pair of processes $(X,Y)$ appearing in Theorem  \ref{Theorem:Main:Dirichlet:Process:Result}
and the process $Z$ itself. Let $D(\mathbb{R}_{+},\mathbb{R}^2)$ denote the space of $\mathbb{R}^2$-valued functions, with domain $\mathbb{R}_{+}$, that are right-continuous with left limits. Also, let $D_S(\mathbb{R}_{+},\mathbb{R}^2)$ be the set of $f \in D(\mathbb{R}_{+},\mathbb{R}^2)$ such that $f(0) \in S$. Next, let $d(\cdot)$ be a set-valued mapping defined on $ S$ such that $d(z)$ is a closed convex cone in ${\mathbb R}^2$ for every $z\in S$. In particular, we define
\begin{eqnarray}
d(z) &=&\begin{cases}
\{a v_1, a \geq 0\}, &\text{for $z \in  \partial S_1\setminus\{0\}$,}\\
\{a v_2, a \geq 0\}, &\text{for $z \in  \partial S_2\setminus\{0\}$,}\\
V, &\text{for}\ z=0,\\
\{0\},&\text{for}\ z\in\text{int}(S),\\
\end{cases} \label{Display:Definition:Of:Direction:Of:Reflection}
\end{eqnarray}
where the selection of the closed convex cone $V\subset{\mathbb R}^2$    will be discussed in a moment. For a set $A \subseteq \mathbb{R}^2$, let $\overline{\textrm{co}}(A)$ denote the closure of the convex hull of $A$. We then recall from Ramanan \cite{ramanan2006reflected} the definition
of the extended Skorokhod problem (ESP).

\begin{definition}\label{Definition:Extended:Skorohod:Problem}
The pair of processes $(\phi,\eta) \in D_S(\mathbb{R}_{+},\mathbb{R}^2) \times D(\mathbb{R}_{+},\mathbb{R}^2)$ solve the ESP $(S,d(\cdot))$ for $\psi \in D_S(\mathbb{R}_{+},\mathbb{R}^2)$ if $\phi(0)=\psi(0)$, and if for all $t \in \mathbb{R}_{+}$,
the following properties hold,
\begin{enumerate}
\item~$\phi(t) ~=~ \psi(t) + \eta(t)$, \label{Definition:ESP:Item:One}
\item~$\phi(t) \in S$, \label{Definition:ESP:Item:Two}
\item For every $s \in [0,t]$, \label{Definition:ESP:Item:Three}
\begin{eqnarray*}
\eta(t) - \eta(s) &\in& \overline{\mathrm{co}}[\cup_{u \in (s,t]}d(\phi(u))],
\end{eqnarray*}
\item $\eta(t)-\eta(t-) \in  \overline{\mathrm{co}}[d(\phi(t))]$. \label{Definition:ESP:Item:Four}
\end{enumerate}
\end{definition}

\begin{theorem}\label{necessary}
Suppose that $1<\alpha<2$. Then, for each $z \in S$, the ESP $(S,d(\cdot))$ for the Brownian motion $X$ on $(C_S,\mathcal{F},\mathcal{F}_t,P^z)$
has a solution $P^z$-a.s. if and only if
\begin{equation}\overline{\mathrm{co}}(V\cup \{a v_1, a \geq 0\} \cup \{a v_2, a \geq 0\} )={\mathbb R}^2.\label{full:space}\end{equation}
In this case, $(Z,Y)$ solves the ESP $(S,d(\cdot))$ for $X$.
\end{theorem}

\noindent By trivially setting $V=\mathbb{R}^2$, it follows that one may always find a $V$ such that \eqref{full:space} holds. However, using the fact that $1 < \alpha < 2$, it is straightforward to verify that the smaller set $V=\{a  v_0, a \geq 0 \}$ for any $v_0$ in the interior of $S$ satisfies \eqref{full:space} as well. Note also that we stop short of claiming in Theorem \ref{necessary} that $(Z,Y)$ is the unique solution to the ESP $(S,d(\cdot))$ for $X$.

\subsection{Organization of the Remainder of the Paper}

For the remainder of the paper, we assume that
\begin{eqnarray*}
&1 < \alpha < 2.&
\end{eqnarray*}
In Section \ref{Section:Excursion:Theory:Background}, we provide some useful results related to the zero set of RBM in a wedge. Much of the material found in Section \ref{Section:Excursion:Theory:Background} is based on the results of Williams \cite{williams1987local}. In Section \ref{Section:Identification:of:X:and:Skorokhod:Problem}, we identify the Brownian motion $X$ which appears in the Doob-Meyer decomposition of $Z$ given by Theorem \ref{Theorem:Main:Dirichlet:Process:Result}. Our proofs of Theorems \ref{Theorem:Main:Dirichlet:Process:Result}, \ref{Theorem:Converse:Dirichlet:Process:Result} and \ref{necessary} may be found in Sections \ref{section:dirichlet:process}, \ref{section:converse:dirichlet} and \ref{section:esp}, respectively.

\section{Results Related to the Zero Set of $Z$}
\label{Section:Excursion:Theory:Background}

The following results are helpful as we proceed throughout the paper. Let $\Lambda = \{t \in \mathbb{R}_{+} : Z(t)=0\}$ denote the zero set of $Z$ and  note that since $Z$ is a continuous process for every $\omega
\in C_S$, it follows that $\Lambda$ is a closed set for every $\omega
\in C_S$. Moreover, since by Lemma 2.2 of Williams \cite{williams1987local} the origin is a regular point for $Z$, using standard arguments (see Theorem 2.9.6 of Karatzas and Shreve \cite{KaratzasShreve}) it may be shown that for each $z \in S$, the set $\Lambda$ contains no isolated points $P^z$-a.s. Next, let $\Lambda^C=\{t \in \mathbb{R}_{+} : Z(t) \neq 0\}$ denote the complement of $\Lambda$ in $\mathbb{R}_{+}$. It then follows that $\Lambda^C$ is an open set in $\mathbb{R}_{+}$, and in fact it may be written as
\begin{eqnarray}
\Lambda^C &=&[0,\tau_0) \cup  \left( \bigcup_{i=1}^{\infty}(G_i,D_i) \right), \label{display:complement:of:lambda}
\end{eqnarray}
where  $\tau_0$ is the first hitting time of the origin by $Z$, i.e.,
\begin{equation}\tau_0=\inf\{t>0: Z(t)=0\}.\label{tau:zero}\end{equation}

 Note that $\tau_0=0$, $P^0$-a.s., in which case we interpret $[0,0)$ as the empty set.
For each $i \geq 1$, we  refer to $(G_i,D_i)$ as an excursion interval of $Z$ away from the vertex and refer to the paths of $Z$ over such intervals as excursions from the vertex. It may easily be shown that the sequence of random times $\{(G_i,D_i), i \geq 1\}$ can be selected such that $(G_i,D_i):C_S \mapsto \mathbb{R}_{+}^2$ is a measurable function on $(C_S,\mathcal{F})$ for each $i \geq 1$. We assume for the remainder of the paper that such a selection has been made.

Closely related to the zero set $\Lambda$ is the concept of the local time of $Z$ at the origin. Since the origin is a regular  point for $Z$, it follows by the results of Blumenthal and Getoor \cite{blumenthal2007markov} that there exists a continuous, non-decreasing, perfect, adapted functional $L$ with support $\{0\}$. We refer to this process as the local time of $Z$ at the origin. Moreover, the results of \cite{blumenthal2007markov}  also guarantee that $L$ is uniquely determined up to a multiplicative constant.  As in \cite{williams1987local}, we assume for the remainder of the paper that the multiplicative constant is chosen such that
\begin{eqnarray*}
E^0 \left[ \int_0^{\infty}e^{-t}dL(t) \right] &=&1.
\end{eqnarray*}

Now note that since $1 < \alpha < 2$, it is straightforward to show using Theorem 2.2 of Varadhan and Williams \cite{varadhan1985brownian}, together with the strong Markov property \cite{varadhan1985brownian} of $Z$, that $L(t) \rightarrow \infty$ as $t \rightarrow \infty$, $P^z$-a.s. for each $z \in $S. Hence, we may $P^z$-a.s. define the inverse local time process
\begin{eqnarray}
L^{-1}(a)&=&\inf\{t \geq 0 : L(t) > a\},~~a \geq 0, \label{def:inverse:local:time}
\end{eqnarray}
and we set $L^{-1}=\{L^{-1}(a), a \geq 0\}$. By Corollary 2.7 of Williams \cite{williams1987local}, under $P^0$, $L^{-1}$ is a stable subordinator of index $\alpha/2$. In particular, $\ln E^0[e^{-\lambda L^{-1}(a)}]=-a\lambda^{\alpha/2}$ for $a, \lambda  \geq 0$. Hence, under $P^0$, $L^{-1}$ is a strictly increasing, right-continuous, pure jump process (see Bertoin \cite{bertoin1999subordinators}), and so we have that $P^0$-a.s.,
\begin{eqnarray*}
L^{-1}(a) &=& \sum_{0 \leq \kappa \leq a} \Delta L^{-1}(\kappa),~a \geq 0, \end{eqnarray*}
where
\begin{eqnarray*}
\Delta L^{-1}(\kappa) &=& L^{-1}(\kappa)-L^{-1}(\kappa -),~\kappa \geq 0,
\end{eqnarray*}
with the convention that $L^{-1}(0-)=0$. In particular, whenever $\Delta L^{-1}(\kappa)\not= 0$, then it is straightforward to show that $\Delta L^{-1}(\kappa)=D_i-G_i$ for the unique value of $i\ge 1$ such that $L(G_i)=L(D_i)=\kappa$. Thus, we also have that $P^0$-a.s.,
\begin{equation}
 L^{-1}(a)~=~\sum_{i:L(G_i)\le a} (D_i-G_i),~a \geq 0.\label{l:inverse}\end{equation}
Finally, we note that using arguments similar to the proof of Corollary 5.11 (page 411) of \c{C}inlar \cite{cinlar2011probability}, it may be shown that $P^0$-a.s., $L^{-1}(\mathbb{R}_{+}) \subset \Lambda $ with
\begin{eqnarray}
\Lambda \backslash  L^{-1}(\mathbb{R}_{+})&=& \bigcup_{i=1}^{\infty}\{G_i\}. \label{Display:Disjoint:Intervals}
\end{eqnarray}

\section{Identification of $X$ and The Standard Skorokhod Problem}
\label{Section:Identification:of:X:and:Skorokhod:Problem}

In this section, we prove the existence of a 2-dimensional process  $X$ on $(C_s, {\mathcal F}, {\mathcal F}_t)$ which is a standard 2-dimensional Brownian motion started from $z$ on $(C_S,\mathcal{F},\mathcal{F}_t,P^z)$ for each $z \in S$. Moreover, letting $Y=Z-X$, we show that over each excursion interval $[G_i,D_i]$ for $i\ge 1$ (and if $z\not= 0$, then also over the interval from zero to the first time that $Z$ hits the origin), the pair of processes $(Z,Y)$ solve the standard Skorokhod problem for $X$ given by  Definition \ref{Definition:Skorohod:Problem} below.
In Section \ref{section:dirichlet:process}, it is proven
that the process $Y$ is of zero energy on $(C_S,\mathcal{F},\mathcal{F}_t,P^z)$ for each $z \in S$, and hence $(X,Y)$
is the desired Doob-Meyer type decomposition of $Z$ given by Theorem \ref{Theorem:Main:Dirichlet:Process:Result}. We also note that our construction of the process $X$ is similar to that of Kang and Ramanan \cite{kang2014submartingale}, but differs in an important way  in order to allow us to prove our results related the Skorokhod problem.

We now recall the definition of the standard Skorokhod problem (SP) as given in Ramanan \cite{ramanan2006reflected}.
Recall first from Section \ref{Section:Extended:Skorokhod:Map:Results} the definition of $D(\mathbb{R}_{+},\mathbb{R}^2)$ as the space of $\mathbb{R}^2$-valued functions, with domain $\mathbb{R}_{+}$, that are right-continuous with left limits, and the definition of $D_S(\mathbb{R}_{+},\mathbb{R}^2)$ as the set of $f \in D(\mathbb{R}_{+},\mathbb{R}^2)$ such that $f(0) \in S$. Also, let $\mathcal{B}V(\mathbb{R}_{+},\mathbb{R}^2) \subset
D(\mathbb{R}_{+},\mathbb{R}^2)$ denote the subset of functions in $D(\mathbb{R}_{+},\mathbb{R}^2)$ which have finite variation on each bounded interval. For each $f \in \mathcal{B}V(\mathbb{R}_{+},\mathbb{R}^2)$
and $t \geq 0$, we denote the variation of $f$ on $[0,t]$ by $\bar f(t)$, which is a shorthand notation for $V_1(f,[0,t])$. Next, recall from \eqref{Display:Definition:Of:Direction:Of:Reflection} of Section \ref{Section:Extended:Skorokhod:Map:Results}, the definition of the set-valued mapping $d(\cdot)$ defined on $S$,
and, for each $z \in S$, let $d^1(z)$ denote the intersection of $d(z)$ with the unit sphere $S_1(0)$ in $\mathbb{R}^2$ which is centered at the origin. We then recall from Ramanan \cite{ramanan2006reflected} the definition
of the standard Skorokhod problem (SP).

\begin{definition}\label{Definition:Skorohod:Problem}
The pair of processes $(\phi,\eta) \in D_S(\mathbb{R}_{+},\mathbb{R}^2) \times D(\mathbb{R}_{+},\mathbb{R}^2)$ solve the SP $(S,d(\cdot))$ for $\psi \in D_S(\mathbb{R}_{+},\mathbb{R}^2)$ if $\phi(0)=\psi(0)$, and if for all $t \in \mathbb{R}_{+}$,
the following properties hold,
\begin{enumerate}
\item~$\phi(t) ~=~ \psi(t) + \eta(t)$,
\item~$\phi(t) \in S$,
\item ~$ \eta\in \mathcal{B}V(\mathbb{R}_{+},\mathbb{R}^2) $,
\item ~$    \bar\eta(t)=  \int_{[0,t]}1\{\phi(s) \in \partial S\}d\bar\eta(s)$  \label{Definition:SP:Item:Four}
\item There exists a measurable function $\gamma:\mathbb{R}_{+} \mapsto S_1(0)$ such that $d\bar\eta$-almost everywhere we have that $\gamma(t) \in d^1(\phi(t))$
and
\begin{eqnarray*}
\eta(t) &=& \int_0^t \gamma(s) d\bar\eta(s).
\end{eqnarray*}\label{Definition:SP:Item:Five}
\end{enumerate}
\end{definition}

Our main result of this section is the following.

\begin{theorem}\label{Proposition:Skorokhod:Problem}Suppose that $1 < \alpha < 2$.
\begin{enumerate}
\item \label{Proposition:Skorohod:Problem:Part:1}

There exists a 2-dimensional process $X$ on $(C_s, {\mathcal F}, {\mathcal F}_t)$ such that $X$ is a standard 2-dimensional Brownian motion started from $z$ on $(C_S,\mathcal{F},\mathcal{F}_t,P^z)$ for each $z \in S$.

\item \label{Proposition:Skorohod:Problem:Part:2}

Let $X$ be as above, $Y=Z-X$, and  let $V$ in \eqref{Display:Definition:Of:Direction:Of:Reflection} be an arbitrary, non-empty closed convex cone in ${\mathbb R}^2$.
Then, for each $z \in S$, $P^z$-a.s. the pair $(Z((G_i+\cdot)\wedge D_i),Y  ((G_i+\cdot)\wedge D_i)-Y(G_i) )$ solves the SP $(S,d(\cdot))$ for $X  ((G_i+\cdot)\wedge D_i) +Y(G_i)$ for each $i \geq 1$. In addition,    for each $z\in S\setminus\{0\}$, $P^z$-a.s. the pair  $(Z(\cdot\wedge \tau_0), Y(\cdot\wedge \tau_0))$ solves the SP $(S,d(\cdot))$ for $X  (\cdot\wedge \tau_0)$.

\end{enumerate}
\end{theorem}

\noindent The two parts of Theorem \ref{Proposition:Skorokhod:Problem} are proven respectively in Sections \ref{Subsection:Proof:That:X:Exists} and \ref{Subsection:Semimartingale:Decomposition:Of:X:On:Excursion:Intervals} below. Also note that $d(0)=V$ in the definition of $d(\cdot)$, and the selection of $V$ turns out to be important in Theorem \ref{necessary}.
 However, since $Z$ does not reach the origin on $(G_i,D_i)$, or on $[0,\tau_0)$ if $z\not= 0$, the selection of $V$ in the above theorem is irrelevant.

\subsection{Proof of Part \ref{Proposition:Skorohod:Problem:Part:1} of Theorem  \ref{Proposition:Skorokhod:Problem} }\label{Subsection:Proof:That:X:Exists}

For each $ \delta > 0$, let $S_\delta \subset S$
be the closed set defined in polar coordinates by $S_\delta = S + \delta e^{i \xi/2}$. Next, set $\tau_0^{\delta}=0$, and, for each $k \geq 1$, recursively define
\begin{eqnarray*}
\sigma^{\delta}_k &=& \inf\{t \geq \tau^{\delta}_{k-1} : Z(t) \in S_{2 \delta}\}~~\textrm{and}~~\tau^{\delta}_k ~=~ \inf\{t \geq \sigma^{\delta}_{k} : Z(t) \in \partial S_{\delta}\}.
\end{eqnarray*}
Since $Z(\cdot,\omega)$ is  continuous for every $\omega\in C_S$, we may equivalently set $\tau^{\delta}_k ~=~ \inf\{t \geq \sigma^{\delta}_{k} : Z(t) \in S_{\delta}^c \cap S^{o}\}$,
where $S_{\delta}^c$ denotes the complement of $S_{\delta}$ in $S$, and $S^{o}$ is the interior of $S$. It is immediate
by Proposition 2.1.5 of Ethier and Kurtz \cite{EK86} that $\sigma^{\delta}_k$
and $\tau^{\delta}_k$ are, for each $k \geq 1$, stopping times relative to the right-continuous filtration $\{\mathcal{F}_t, t \geq 0\}$.

Now, for each $k \geq 1$, let $W_{(k)}^{\delta}$ be the process defined by setting
\begin{eqnarray*}
W_{(k)}^{\delta}(t)&=& Z(t \wedge \tau_{k}^{\delta})-Z(t \wedge \sigma_{k}^{\delta}),~t \geq 0,
\end{eqnarray*}
and then define the process $W^{\delta}$ by setting
\begin{eqnarray*}
W^{\delta}(t) &=& \sum_{k=1}^{\infty}W_{(k)}^{\delta}(t),~t \geq 0.
\end{eqnarray*}

Recall next from Section \ref{Section:Introduction} the definition of the augmented sigma fields ${\mathcal M}^z$ and ${\mathcal M}_t^z$ for $z\in S$ and $t \geq 0$. For each $z \in S$, let $\Upsilon_2^{c,z}$ denote the space of continuous, square-integrable martingales $M$ with time horizon $\mathbb{R}_{+}$ on $(C_S,\mathcal{M}^z,\mathcal{M}^z_t,P^z)$, and such that $M(0)=0$.  It is well-known (see for instance Proposition
1.5.23 of Karatzas and Shreve \cite{KaratzasShreve})  that the space $\Upsilon_2^{c,z}$ is complete under the norm
\begin{eqnarray}
\| M \|_z &=& \sum_{n=1}^{\infty} \frac{1}{2^n}\left(\sqrt{E^z[M^2(n)]} \wedge 1\right),~ M \in \Upsilon_2^{c,z}. \label{norm:def}
\end{eqnarray}

Now, for each $\delta > 0$ and $k \geq 1$ and $i=1,2,$ denote the $i$th component process of $W_{(k)}^{\delta}$ by $W_{(k),i}^{\delta}$.  We begin our proof of Part \ref{Proposition:Skorohod:Problem:Part:1} of Theorem \ref{Proposition:Skorokhod:Problem} by recalling a result from Kang and Ramanan \cite{kang2014submartingale}.

\begin{lemma}\label{Lemma:Xk:Martingale}
For each $\delta > 0$, $k \geq 1,i=1,2$, and $z\in S$, the process $W_{(k),i}^{\delta}  \in \Upsilon_2^{c,z}$  with quadratic variation processes
\begin{eqnarray*}
\langle W_{(k),i}^\delta \rangle_t &=& (t \wedge \tau_k^\delta) - (t \wedge \sigma_k^\delta)~,t \geq 0,
\end{eqnarray*}
for $i=1,2,$ and $\langle W_{(k),1}^\delta,W_{(k),2}^\delta \rangle_t = 0, t \geq 0$.
\end{lemma}

\begin{proof}
This follows from Lemma 5.1 in \cite{kang2014submartingale}.
\end{proof}

Next, we prove that $W_{(k)}^{\delta}$ and $W_{(\ell)}^{\delta}$ and orthogonal to one another for $k \neq \ell$.

\begin{lemma}\label{Lemma:Xk:Xl:Orthogonal}For each $\delta > 0$ and $k, \ell \geq 1$ with $k \neq \ell$, we have under $P^z$ for every $z\in S$ that
\begin{eqnarray*}
\langle W_{(k),i}^{\delta},W_{(\ell),j}^{\delta} \rangle_t &=& 0,~ t \geq 0,
\end{eqnarray*}
for $i,j=1,2$.
\end{lemma}

\begin{proof}Let $\delta > 0$ and assume without loss of generality that $\ell > k \geq 1$. The result now follows from the fact that $W_{(k)^,i}^\delta$ is flat outside of $(\sigma_k^\delta,\tau_k^\delta)$ for $i=1,2,$ and
$W_{(\ell),i}^\delta$ is flat outside of $(\sigma_{\ell}^\delta,\tau_{\ell}^\delta)$ for $i=1,2,$ together with the fact
that $(\sigma_k^\delta,\tau_k^\delta) \cap (\sigma_{\ell}^\delta,\tau_{\ell}^\delta) = \emptyset $.
Indeed, let $\{\pi_n=\{0=t^n_0 < t^n_1 < ... < t^n_{s(n)}=t\}, n \geq 1 \}$ be such that $\| \pi_n \| \rightarrow 0$ as $n \rightarrow \infty$,
and let
\begin{eqnarray*}
C_{k,i,\ell,j}(\pi_n) &=& \sum_{m=1}^{s(n)}(W_{(k),i}^\delta(t^n_m)-W_{(k),i}^\delta(t^n_{m-1}))(W_{(\ell),j}^\delta(t^n_m)-W_{(\ell),j}^\delta(t^n_{m-1}))
\end{eqnarray*}
for $i,j=1,2$. Then, $C_{k,i,\ell,j}(\pi_n)=0$ whenever $\| \pi_n \| < \sigma_l^\delta - \tau_k^\delta$. Thus, $P^z$-a.s. for each $z \in S$ we have $C_{k,i,\ell,j}(\pi_n) \rightarrow 0$
as $n \rightarrow \infty$.
\end{proof}

\begin{lemma}\label{Wdelta}
For each $z\in S$,\ $\delta>0$ and $i=1,2,$ we have that $W_i^\delta\in\Upsilon^{c,z}_2$ with quadratic and cross variations
\begin{equation}\langle W^{\delta}_i \rangle_t = \sum_{k=1}^{\infty}((t \wedge \tau_k^\delta)-(t \wedge \sigma_k^\delta)),~t \geq 0,\label{qv}\end{equation}
and $\langle W_1^\delta,X_2^\delta\rangle_t=0,t \geq 0$.
\end{lemma}

\begin{proof}

Let $\delta > 0$ and $z \in S$. We first identify the quadratic variation processes of $W^{\delta}$. Note that for each $n \geq 1$ we may write
\begin{eqnarray*}
W^{\delta}(t \wedge \tau_n^\delta) &=& \sum_{k=1}^{n}W_{(k)}^\delta(t),~t \geq 0,
\end{eqnarray*}
which by Lemma \ref{Lemma:Xk:Martingale} implies that $W^{\delta}$ is a local martingale on $(C_S,{\mathcal M},{\mathcal M}_t,P^z)$. Also by Lemma \ref{Lemma:Xk:Martingale} it
follows that $W^{\delta}(\cdot \wedge \tau_n^\delta) \in \Upsilon_2^{c,z}$. For $n \geq 1$ let
\begin{eqnarray*}
V_n(t) &=& \sum_{k=1}^{n}((t \wedge \tau_k^\delta)-(t \wedge \sigma_k^\delta))~~\textrm{and}~~V(t) ~=~ \sum_{k=1}^{\infty}((t \wedge \tau_k^\delta)-(t \wedge \sigma_k^\delta)),~t \geq 0.
\end{eqnarray*}
By Lemmas \ref{Lemma:Xk:Martingale} and \ref{Lemma:Xk:Xl:Orthogonal} and the definitions of $V_n(t)$ and $V(t)$ it now follows that
\begin{eqnarray*}
\langle W^{\delta}_i(\cdot \wedge \tau_n^\delta) \rangle_t &=&\sum_{k=1}^{n} \langle W^{\delta}_{(k),i} \rangle_t ~=~V_n(t)~=~V(t \wedge \tau_n^\delta),~t \geq 0,
\end{eqnarray*}
for $i=1,2,$ and hence $\{(W^{\delta}_i(t \wedge \tau_n^\delta))^2-V(t \wedge \tau_n^\delta), t \geq 0 \}$ is a continuous martingale on  on $(C_S,{\mathcal M},{\mathcal M}_t,P^z)$
which implies \eqref{qv}. Also since $\langle W^{\delta}_i \rangle_t=V(t) \leq t$ we
have that $W^{\delta}_i \in \Upsilon_2^{c,z}$ for $i=1,2$.  {The proof that $\langle W^{\delta}_1,W^{\delta}_2 \rangle_t=0,~t \geq 0$,
follows in a similar manner.}
\end{proof}

Let $\Gamma$ be the set of sequences $\{a(n), n\ge 1\}$ such that $a(n)\downarrow 0$ as $n\to\infty$ and $2a(n+1)<a(n)$ for every $n\ge 1$.

\begin{lemma}
For every  sequence $\{\delta(n),n\ge 1\}\in\Gamma$ we have that  $\{W^{\delta(n)}_i, n\ge 1 \}$ is Cauchy in $\Upsilon^{c,z}_2$ for every $z\in S$ and $i=1,2$.
\end{lemma}

 \begin{proof} Let $\{\delta(n), n \geq 1\} \in\Gamma$.  Clearly it is sufficient to prove that $\{W_i^{\delta(n)}(T), n \geq 1\}$ is Cauchy
in $L^2(C_S,{\mathcal F}, P^z)$ for each $T \geq 0$. Suppose that for each $n > m \geq 1$ we  show that
\begin{equation}
 E^z[(W_i^{\delta(n)}(T)-W_i^{\delta(m)}(T))^2 ] \leq E^z \int_0^T 1\{Z(t) \in S \backslash S_{2 \delta (m)}\}dt. \label{Display:Inequality:Using:Integral}
\end{equation}
Then by the Dominated Convergence Theorem  and Lemma 4.2 of Williams \cite{williams1985recurrence} we have that
\begin{eqnarray}
 \lim_{m \rightarrow \infty} E^z \left[ \int_0^T 1\{Z(t) \in S \backslash S_{2 \delta (m)}\}dt \right] ~=~E^z \left[ \int_0^T 1\{Z(t) \in \partial S \} dt
 \right] &=&0, \label{Display:Integral:Goes:To:Zero}
\end{eqnarray}
and so $\{W_i^{\delta(n)}(T), n \geq 1\}$ is Cauchy in $L^2(C_S,{\mathcal F}, P^z)$ as desired.

Hence we need to show only \eqref{Display:Inequality:Using:Integral}, and that is what we shall do in the rest of this proof. We first claim that for each $k \geq 1$ and $n > m \geq 1$ there exists an $\ell \geq 1$ such that $[\sigma_k^{\delta(m)},\tau_k^{\delta(m)}] \subseteq  [\sigma_{\ell}^{\delta(n)},\tau_{\ell}^{\delta(n)}] $.
Suppose that $t \in [\sigma_k^{\delta(m)},\tau_k^{\delta(m)}] $. Then $Z(t) \in S_{\delta(m)}$. Moreover since $2 \delta(n) < \delta(m)$, if
$Z(t) \in S_{\delta(m)}$ then $t \in [\sigma_{\ell}^{\delta(n)},\tau_{\ell}^{\delta(n)}] $ for some $\ell \geq 1$. We
now claim that the choice of $\ell \geq 1$ does not depend on the choice of $t \in [\sigma_k^{\delta(m)},\tau_k^{\delta(m)}]$.
Suppose that $t_1, t_2 \in [\sigma_k^{\delta(m)},\tau_k^{\delta(m)}]$ with $t_1 \in [\sigma_{\ell}^{\delta(n)},\tau_{\ell}^{\delta(n)}]$
and $t_2 \in [\sigma_{p}^{\delta(n)},\tau_{p}^{\delta(n)}]$ for some $p > \ell$. Then $\sigma_p^{\delta(n)} \in [\sigma_k^{\delta(m)},\tau_k^{\delta(m)}]$ which is impossible since $Z(\sigma_p^{\delta(n)}) \in \partial S_{2 \delta(n)}$
and so the result follows.

Now write
\begin{equation*}
E^z[(W_i^{\delta(n)}(T)-W_i^{\delta(m)}(T))^2 ]
= E^z[\langle W_i^{\delta(n)}-W_i^{\delta(m)} \rangle_T ]\end{equation*}
\begin{eqnarray*}&=&E^z \left[ \sum_{k=1}^{\infty} \left[ \langle W_i^{\delta(n)}-W_i^{\delta(m)} \rangle_{\tau_k^{\delta(m)} \wedge T}- \langle W_i^{\delta(n)}-W_i^{\delta(m)} \rangle_{\sigma_k^{\delta(m)} \wedge T}   \right] \right]\\
&+&E^z \left[ \sum_{k=1}^{\infty} \left[ \langle W_i^{\delta(n)}-W_i^{\delta(m)} \rangle_{\sigma_k^{\delta(m)} \wedge T}- \langle W_i^{\delta(n)}-W_i^{\delta(m)} \rangle_{\tau_{k-1}^{\delta(m)} \wedge T}   \right] \right].
\end{eqnarray*}
We analyze each of the terms on the right hand side above separately. Regarding the first term recall that for each $k \geq 1$ and $n > m \geq 1$ there exists an $\ell \geq 1$ such that $[\sigma_k^{\delta(m)},\tau_k^{\delta(m)}] \subseteq  [\sigma_{\ell}^{\delta(n)},\tau_{\ell}^{\delta(n)}] $. It
then follows that for $k \geq 1$ and $n > m$,
\begin{eqnarray*}
 W_i^{\delta(m)}(T \wedge \tau_k^{\delta(m)})- W_i^{\delta(m)}(T \wedge \sigma_k^{\delta(m)})&=& W_i^{\delta(n)}(T \wedge \tau_k^{\delta(m)})- W_i^{\delta(n)}(T \wedge \sigma_k^{\delta(m)}),
\end{eqnarray*}
hence by Proposition IV.1.13 in \cite{revuzyor} the first term is zero.
Regarding the second term note that by \eqref{qv} for every $\delta>0$, $0\le s<t$, and $i=1,2$ we have
\begin{equation*}\langle W^{\delta}_i \rangle_t - \langle W^{\delta}_i \rangle_s \le t-s,
\end{equation*}
and on the interval $[\tau^{\delta(m)}_{k-1}, \sigma_k^{\delta(m)}]$ the process $W^{\delta(m)}$ is flat. Then it follows that
 for $i=1,2,$ we have
\begin{equation*}
 \sum_{k=1}^{\infty} \left[ \langle W_i^{\delta(n)}-W_i^{\delta(m)} \rangle_{\sigma_k^{\delta(m)} \wedge T}- \langle W_i^{\delta(n)}-W_i^{\delta(m)} \rangle_{\tau_{k-1}^{\delta(m)} \wedge T}   \right]
= \sum_{k=1}^{\infty} \left[ \langle W_i^{\delta(n)}\rangle_{\sigma_k^{\delta(m)} \wedge T}- \langle W_i^{\delta(n)}\rangle_{\tau_{k-1}^{\delta(m)} \wedge T}   \right] \end{equation*}
\begin{equation*}\le\sum_{k=1}^{\infty}  \left[ (\sigma_k^{\delta(m)} \wedge T) - (\tau_{k-1}^{\delta(m)} \wedge T)   \right].
\end{equation*}
Now note that if $t \in [\tau_{k-1}^{\delta(m)},\sigma_k^{\delta(m)}]$ for some $k \geq 1$ then $Z(t) \in S \backslash S_{2\delta(m)}$, thus
 \begin{equation}
 \sum_{k=1}^{\infty}  \left[ (\sigma_k^{\delta(m)} \wedge T) - (\tau_{k-1}^{\delta(m)} \wedge T)   \right]=\int_0^T  \sum_{k=1}^{\infty} 1\{t \in [\tau_{k-1}^{\delta(m)},\sigma_{k}^{\delta(m)} ]\} dt
 \leq \int_0^T 1\{Z(t) \in S \backslash S_{2 \delta (m)}\}dt,  \label{bound}
\end{equation}
and  \eqref{Display:Inequality:Using:Integral} follows.
\end{proof}

\begin{proposition}\label{Proposition:Brownian:Motion}
There exists a process $W$ on $(C_S,{\mathcal F}, {\mathcal F_t})$ such that for every sequence $\{\delta(n),n\ge 1\}\in\Gamma$ and every $z\in S$, the process $W$ is a standard 2-dimensional Brownian motion under $P^z$ started at zero, and
\begin{equation} \lim_{n\to\infty} W_i^{\delta(n)}=W_i\label{L2limit},~i=1,2,\end{equation}
in the norm topology given in \eqref{norm:def}.

\end{proposition}

\begin{proof}
 {Let $z \in S$ and for $i=1,2,$ let $W_i^z$ be the unique limit point of the Cauchy sequence $\{W_i^{\delta(n)}, n \geq 1\}$ in the complete space $\Upsilon_2^{c,z}$. We claim that $W^z=(W_1^z,W_2^z)$ is a 2-dimensional standard Brownian motion on $(C_S,{\mathcal M},{\mathcal M}_t,P^z)$.} Since $W^z \in {M}_2^{c,z}$ it suffices to characterize the finite-dimensional distributions of $W^z$.  By Lemma \ref{Wdelta}, \eqref{bound},
 \eqref{Display:Integral:Goes:To:Zero}, and the Martingale Invariance Principle (see Theorem 7.1.4 of \cite{EK86}) it follows that $W^{\delta(n)} \Rightarrow W^0$ as $n \rightarrow \infty$ under $P^z$, where $W^0$ is a standard 2-dimensional Brownian motion. This implies that for each $m \geq 1$ and
$0 \leq t_1 < t_2 < ... < t_m$, $(W^{\delta(n)}(t_1),...,W^{\delta(n)}(t_m)) \Rightarrow (W^{0}(t_1),...,W^{0}(t_m))$
as $n \rightarrow \infty$. On the other hand, for each $i=1,2,$ the sequence $\{W_i^{\delta(n)}, n \geq 1\}$ is Cauchy in $\Upsilon_2^{c,z}$
with unique limit point $W_i^z$ and so $W_i^{\delta(n)}(t_j) \rightarrow W_i^z(t_j)$ in $L^2(P^z)$ as $n \rightarrow \infty$
for each $j=1,...,m$. Thus there exists a subsequence $\{\delta(n(k)), k \geq 1\}$ such that $(W^{\delta(n(k))}(t_1),...,W^{\delta(n(k))}(t_m)) \rightarrow (W^z(t_1),...,W^z(t_m)), P^z$-a.s. as $k \rightarrow \infty$. But this then implies
that $W^z$ and $W^0$ have the same finite dimensional distributions, thus $W^z$ is indeed a standard Brownian motion under $P^z$. We still need to show the existence of a process $W$ on $(C_S,{\mathcal F}, {\mathcal F}_t)$  such that for every $z\in S$ the process $W^z$ is indistinguishable from $W$ under $P^z$. Notice that by \eqref{L2limit} for every $t\ge 0$ we have that $W^{\delta(n)}(t)\to W^z (t)$ in probability as $n\to \infty$, and the existence of such $W$ follows from Cinlar, Jacod, Protter, and Sharpe \cite{cinlar1980semimartingales}, Lemma 3.29. Note also that Hypothesis 3.1 and condition (iv) at the beginning of Section 3a in \cite{cinlar1980semimartingales}  are satisfied by \cite{varadhan1985brownian}, Corollary 3.3 and Theorem 3.14.
\end{proof}

\begin{proof}[Proof of Part \ref{Proposition:Skorohod:Problem:Part:1} of Theorem \ref{Proposition:Skorokhod:Problem}]
We define
\begin{equation} X(t,\omega)=W(t,\omega)+\omega(0),\ t\ge 0,\label{X:def}\end{equation}
and note that by Proposition \ref{Proposition:Brownian:Motion} for every $z\in S$ the 2-dimensional process $X$ is indeed a standard Brownian motion on $(C_S,{\mathcal F}, {\mathcal F_t}, P^z)$ started at $z$, exactly as required.
\end{proof}

\subsection{Proof of Part \ref{Proposition:Skorohod:Problem:Part:2} of Theorem \ref{Proposition:Skorokhod:Problem}}
\label{Subsection:Semimartingale:Decomposition:Of:X:On:Excursion:Intervals}

Let the set of continuous functions $\omega:[0,\infty)\mapsto{\mathbb R}^2$   be denoted by $C_{\mathbb{R}^2}$, which is an extension of $C_S$.
 On this space, for each $t\ge s \geq 0$, let $\tilde{\mathcal M}_t=\sigma\{\omega(s):0\le s\le t\}$, the sigma-field of subsets of $C_{{\mathbb R}^2}$ generated by the coordinate maps $\omega\mapsto\omega(s)$ for $0\le s\le t$. Similarly, define $\tilde{\mathcal M}=\sigma\{\omega(s): s\in[0,\infty) \}$. As explained in Williams \cite{williams1985reflected}, these $\sigma$-fields represent the  natural extensions of  ${\mathcal M}_t$ and ${\mathcal M}$ from $C_S$ to $C_{\mathbb{R}^2}$.
Let $\tilde X:[0,\infty)\times C_{\mathbb{R}^2}\mapsto\mathbb{R}^2$ be the coordinate-mapping process
\begin{eqnarray*}
\tilde X(t,\omega)&=&\omega(t),~~ t\ge 0,~w\in C_{\mathbb{R}^2}.
\end{eqnarray*}

For each $z\in S$, let $Q^z$ be a probability measure on $\tilde{\mathcal M}$ such that $\tilde X$ is a Brownian motion starting at $z$ under $Q^z$. Now let $R$ be the $2\times 2$ matrix whose column vectors are $v_1$ and $v_2$.
The set $\mathcal A$ defined below will play a central role in this subsection.

\begin{definition}\label{Definition:Of:The:Set:A}
Let ${\mathcal A}\subset C_S\times C_{\mathbb{R}^2}$ be the set that consists of all pairs $(z,y)\in C_S\times C_{\mathbb{R}^2}$ satisfying the following two properties:
\begin{enumerate}
\item There exists $u\in C_{\mathbb{R}^2}$ such that both components of $u$ are non-decreasing, and y(t)=Ru(t) for all $t\ge 0$, \label{Item:One:The:Set:A}
\item For $j=1,2,$ the $j^{ th  }$ component $u_j$ of $u$ increases only at times $t \geq 0$ when $z(t)\in\partial S_j$, i.e., $\int_0^\infty 1\left\{z(v)\in S\setminus\partial S_j\right\}du_j(v)=0$. \label{Item:Two:The:Set:A}
\end{enumerate}
\end{definition}
\noindent We then have the following result.

\begin{lemma}
The set ${\mathcal A}$ is measurable, that is, ${\mathcal A}\in{\mathcal M}\times\tilde{\mathcal M}$.
\end{lemma}

\begin{proof} Let $\tilde {\mathcal A}\subset C_S\times C_{\mathbb{R}^2}$ be a set such that $(z,u)\in\tilde {\mathcal A}$ if both components of $u$ are non-decreasing, and part \ref{Item:Two:The:Set:A} of the above definition holds. We then have ${\mathcal A}=\{(z,Ru):(z,u)\in\tilde {\mathcal A}\}$ and thus it is sufficient to show that $\tilde {\mathcal A}$ is a measurable set.
{It is shown in Stroock and Varadhan \cite{stroock2007multidimensional}}, page 30, that ${\mathcal M}\times\tilde {\mathcal M}$ is exactly the class of Borel sets in $C_S\times C_{\mathbb{R}^2}$ under the topology of uniform convergence on compact sets. It is therefore sufficient to show that $\tilde {\mathcal A}$ is closed under the above topology. In order to show this,  let $\{(z_n,u_n),n\ge 1\}\subset \tilde {\mathcal A}$ be such that
$(z_n,u_n)\to (z,u)$ uniformly on compact sets. Since both components of $u_n$ are non-decreasing for each $n
\geq 1$, both components of $u$ must be non-decreasing as well. Next we show that $(z,u)$ satisfies Condition \ref{Item:Two:The:Set:A} of Definition \ref{Definition:Of:The:Set:A}. Let $t > 0$ such that $z(t)\notin \partial S_j$, where $j$ is either 1 or 2. {We need to show that there exists a neighborhood $(a,b)$ of $t$ such that $u_j$ is flat on $(a,b)$, where $u_j$ is the $j^{\rm th}$ component of $u$. From $z(t)\notin \partial S_j$ follows that there exists a closed ball $B$ centered around $z(t)$ such that $B\cap \partial S_j=\emptyset$. By the continuity of $z$  there exists a neighborhood $(a,b)$ of $t$ such that $z(q)\in B$ for all $q\in(a,b)$. Then by the uniform convergence of $z_n$ to $z$ there exists another closed ball $C\supset B$ centered also around $z(t)$ and $N\in{\mathbb N}_+$such that $C\cap \partial S_j=\emptyset$,   and for all $n\ge N$ we have $z_n(q)\in C$ for all $q\in(a,b)$. But $(z_n,u_n)\in \tilde {\mathcal A}$ implies that $u_{n,j}$ (the $j^{\rm th}$ component of $u_n$) must be flat on $(a,b)$ for $n\ge N$, hence $u_j$ is also flat on $(a,b)$.} Finally, if $t=0$ and $z(t)\notin \partial S_j$, then we already saw that there exists $b>0$ such that $u_j$ is flat on $(0,b)$. But then the continuity of $u_j$ implies that it is not increasing in $t=0$.
\end{proof}

The next proposition is a restatement   of the results on page 163 and of Theorem 1 of Williams \cite{williams1985reflected}.

\begin{proposition}\label{williams}
There exists a unique pair $(\tilde Z, \tilde Y)$ of continuous, adapted processes on $(C_{\mathbb{R}^2},\tilde{\mathcal M},\tilde{\mathcal M}_t)$
(both $\tilde Z$ and $\tilde Y$ are two-dimensional) such that\\
\begin{equation} \tilde Y(0)=0,\ \tilde Z(t)\in S\ \hbox{for each}\ t\ge 0,\label{A4}\end{equation}
\begin{equation}\tilde Y(t\vee\tilde\tau_0)=\tilde Y(\tilde\tau_0),\quad t\ge 0,\label{A4.1}\end{equation}
\begin{equation}
\tilde Z(t)=\begin{cases} \tilde X(t)+\tilde Y(t) &\mbox{if }t\le\tilde\tau_0,\\
0 &\mbox{if } t\ge \tilde\tau_0,
\end{cases}\label{A5}
\end{equation}
where $\tilde \tau_0=\inf\{t\ge 0:\tilde Z(t)=0\}$,
\begin{equation*} (\tilde Z,\tilde Y)\in {\mathcal A}.
\end{equation*}
Furthermore, for every $z\in S$, $Z(\cdot\wedge\tau_0)$ has the same law under $P^z$ as $\tilde Z(\cdot\wedge\tilde \tau_0)=\tilde Z(\cdot)$ has under $Q^z$, i.e., for any $B\in{\mathcal M}$
\begin{equation*}P^z\left(Z(\cdot\wedge\tau_0)\in
 B\right) = Q^z\left(\tilde Z(\cdot\wedge\tilde\tau_0)\in B\right).
 \end{equation*}
\end{proposition}

Recall now the definitions of the stopping times $\tau_k^{\delta},\sigma_k^\delta$ from the beginning of Section \ref{Section:Identification:of:X:and:Skorokhod:Problem}. Since $Z$ is the coordinate-mapping process, these definitions could be cast in the form $\sigma_k^\delta(w)=\inf\{t\ge \tau_{k-1}^\delta: w(t)\in S_{2\delta}\}$ and $\tau_k^\delta(w)=\inf\{t\ge \sigma_k^\delta: w(t)\in \partial S_{\delta}\}$ for $w\in C_S$. We can also write $\tau_0=\inf\{t>0: w(t)=0\}$ for $w\in C_S$. In principle, some of these stopping times may be infinity for some $w\in S$, but by the continuity of $\omega$ we have
$\lim_{k\to\infty}\sigma_k^\delta(\omega) = \lim_{k\to\infty}\tau_k^\delta(\omega) = \infty$ for all $\omega\in C_S$.
For $\delta>0$,  we define the measurable mapping $I_{\delta} : C_S\mapsto C_{\mathbb{R}^2}$ as
\begin{equation*} I_{\delta}(\omega)(\cdot)=\sum_{k=1}^\infty \left[ \omega\left(\cdot\wedge\tau_k^\delta(\omega)\right) -
\omega\left(\cdot\wedge\sigma_k^\delta(\omega)\right)\right].
\end{equation*}
By the continuity of $\omega$, for any $t\ge 0$, in the infinite sum defining $I_{\delta}(\omega)(t)$  all terms except finitely many are zero.
We also note that
\begin{equation}I_{\delta}(\omega)(\cdot)=W^\delta (\cdot,\omega).\label{I}\end{equation}

Now let  $\underline 0$ be the zero function in $C_{\mathbb{R}^2}$.

\begin{definition} For any given sequence $\{\gamma(n),n\ge 1\} \in\Gamma$,     we define the mapping $F_{\gamma}:C_S\mapsto C_{\mathbb{R}^2}$ in the following way: if there exists a function $\tilde \omega \in C_{\mathbb{R}^2}$ such that for each $m\in{\mathbb N}_+$,
\begin{equation*}\sup_{t\le m}\|I_{\gamma(n)}(\omega)(t)-\tilde \omega(t)\|^2 \to 0 \hbox{ as } n\to\infty,
\end{equation*}
then $F_{\gamma}(\omega)=\tilde\omega$. If such $\tilde\omega$ does not exist, then let $F_{\gamma}(\omega)$ be $\underline 0$.
\end{definition}

Notice that for $\delta>0$, $T\ge 0$, and $\omega\in C_S$,
\begin{equation*} I_{\delta}(\omega)(\cdot\wedge T)=I_{\delta}(\omega(\cdot\wedge T)),
\end{equation*}
hence  for all $\omega\in C_S$ such that   $F_\gamma(\omega)\not=\underline 0$, we have
\begin{equation}F_\gamma(\omega)(\cdot\wedge T)=F_\gamma(\omega(\cdot\wedge T)).\label{stopping}\end{equation}

\begin{lemma}
 $F_{\gamma}$ is ${\mathcal M}/\tilde{\mathcal M}$ measurable.

\end{lemma}
\begin{proof}This follows immediately from \cite{brown:pearcy:operator:theory}, page 111, point R.\end{proof}

\begin{lemma}\label{shift:I} Let $\delta>0$, $T\ge 0$, $\omega\in C_S$ such that
\begin{equation}\omega(T)\in S_{2\delta}^0\cup(S\setminus S_\delta),\label{condition}\end{equation}
where $S_{2\delta}^0$ is the interior of $S_{2\delta}$.
Then
\begin{equation} I_\delta(\omega(T+\cdot))(t)= I_\delta(\omega)(T+t)-I_\delta(\omega)(T),\ t\ge 0.\label{I:shifted}\end{equation}
\end{lemma}

\begin{proof} Let $\tau_0^{\delta,T}(\omega)=T$, $\sigma_k^{\delta,T}(\omega)=\inf\{t\ge \tau_{k-1}^{\delta,T}(\omega):\omega(t)\in S_{2\delta}\}$ and
$\tau_k^{\delta,T}(\omega)=\inf\{t\ge \sigma_{k}^{\delta,T}(\omega):\omega(t)\in \partial S_{\delta}\}$ for $k\ge 1$.
Condition \eqref{condition} implies that there exists $n\in{\mathbb N}_+$ such that $\sigma_n^\delta(\omega)\le \sigma_1^{\delta,T}(\omega)$, $\tau_n^\delta(\omega)=\tau_1^{\delta,T}(\omega)$, and $[\sigma_{n+k}^\delta(\omega),\tau_{n+k}^\delta(\omega)]=[\sigma_{k+1}^{\delta,T}(\omega),\tau_{k+1}^{\delta,T}(\omega)]$ for $k\ge 1$ (this can be checked by looking at the cases $\omega(T)\in S_{2\delta}^0$ and $\omega(T)\in S\setminus S_\delta$ separately). Then clearly both sides of \eqref{I:shifted} are equal to
$$\sum_{k=n+1}^\infty \left\{\omega(\tau_k^\delta(\omega)\wedge t) - \omega(\sigma_k^\delta(\omega)\wedge t)\right\} + \omega(\tau_n^\delta(\omega)\wedge t) -\omega(T\vee\sigma_n^\delta).$$

\end{proof}

\begin{lemma} For every sequence $\gamma\in\Gamma$, $T\in[0,\infty)$, and $\omega\in C_S$ such that $F_\gamma(\omega)\not=\underline 0$, we have
\begin{equation}F_\gamma(\omega(T+\cdot))=F_\gamma(\omega)(T+\cdot) - F_\gamma(\omega)(T).\label{F:shifted}\end{equation}
\end{lemma}

\begin{proof}
For sufficiently large $n$ we have that $\omega(T)\in S_{2\delta(n)}^0\cup(S\setminus S_\delta)$, hence by Lemma \ref{shift:I} for these values of $n$ we have
\begin{equation} I_{\gamma(n)}(\omega(T+\cdot))= I_{\gamma(n)}(\omega)(T+\cdot)-I_{\gamma(n)}(\omega)(T).\label{shifted:n}
\end{equation}
By our assumption that $F_\gamma(\omega)\not=\underline 0$ we have that $I_{\gamma(n)}(\omega)(T+\cdot)\to F_\gamma(\omega)(T+\cdot)$ as $n\to\infty$, in the topology of uniform convergence on compacts. Also, $I_{\gamma(n)}(\omega)(T)\to F_\gamma(\omega)(T)$ as $n\to\infty$. Then by \eqref{shifted:n} $I_{\gamma(n)}(\omega(T+\cdot))$ also converges as $n\to \infty$ in the topology of uniform convergence on compact sets, and by the definition of $F_\gamma$, the limit must be $F_\gamma(\omega(T+\cdot))$.
\end{proof}

The relationship \eqref{F:shifted} can be cast in the form
\begin{equation}
F_\gamma(\omega)(t)-F_\gamma(\omega)(s)=F_\gamma(\omega(s+\cdot))(t-s),\ 0\le s \le t.\label{F:additive}\end{equation}
That is, $F_\gamma$ is additive on $\{\omega\in C_S: F_\gamma(\omega\not=\underline 0\}$. The combination of \eqref{F:shifted} and \eqref{stopping} yields that for $0\le T_1\le T_2$, $\omega\in C_S$, and $\{\gamma(n),n\ge 1\}\in\Gamma$ such that $F_\gamma(\omega)\not=\underline 0$, we have
\begin{equation}F_\gamma(\omega((T_1+\cdot)\wedge T_2))(t)=F_\gamma(\omega)((T_1+t)\wedge T_2) - F_\gamma(\omega)(T_1).\label{shiftedandstopped}\end{equation}

\begin{lemma} Let $a,\lambda>0$, $\{\gamma(n),n\ge 1\}\in\Gamma$ and $\omega\in C_S$ be such that $F_\gamma(\omega)\not=\underline 0$. Then
\begin{equation}aF_\gamma(\omega)(\lambda\cdot)
=F_{a\gamma}(a\omega(\lambda\cdot)),\label{linear}\end{equation}
where $a\gamma = \{a\gamma(n),n\ge 1\}$.

\end{lemma}

\begin{proof}
One can easily see that
\begin{equation*}aI_{\gamma(n)}(\omega)(\lambda\cdot)=I_{a\gamma(n)}(a\omega(\lambda\cdot)),\ n\ge 1.\end{equation*}
But the assumption $F_\gamma(\omega)\not=\underline 0$ implies that the expression on left-hand side converges to $aF_\gamma(\omega)(\lambda\cdot)$ as $n\to\infty$ in the topology of uniform convergence on compacts. Then the expression on the right-hand side also converges, and by the definition of $F_{a\gamma}$ the limit must be $F_{a\gamma}(a\omega(\lambda\cdot))$.
\end{proof}

Now recall the process $W$ from Proposition  \ref{Proposition:Brownian:Motion}  which also appears implicitly in \eqref{decomposition} via \eqref{X:def}. The following lemma is instrumental for the developments in this section. It specifies the exact form of the dependence of $W(\cdot,\omega)$ on $\omega$.

\begin{lemma}\label{sub} For every $z\in S$  and $\{\delta(n),n\ge 1\}\in\Gamma$ there exists a sequence $\{\gamma^z(n),n\ge 1\}\subset \{\delta(n),n\ge 1\}$ such that
\begin{equation}P^z\left(F_{\gamma^z}(\omega)=W(\cdot,\omega)\right)=1.\label{gamma:z}\end{equation}
In addition, the following three statements hold:\\
~\\
(a) If for some $z\in S$ and $\{\gamma(n), n\ge 1\}\in \Gamma$ we have $P^z(F_\gamma(\omega)=\underline 0)=0$ then
\begin{equation} P^z(F_\gamma(\omega)=W(\cdot,\omega))=1;\label{more}\end{equation}
(b) For every stopping time $\tau$ on $(C_s,{\mathcal F},{\mathcal F}_t)$ and $z\in S$,
\begin{equation}E^z\left[P^{Z(\tau)}\left(F_{\gamma^z}(\omega)=W(\cdot,\omega)\right)\right]=1;\label{stoppingtime}\end{equation}
(c) For every $z\in S$,
\begin{equation} P^0\left(F_{\gamma^z}(\omega)=W(\cdot,\omega)\right)=1.\label{z:is:zero}\end{equation}

\end{lemma}

\begin{proof}
Let $\{\delta(n),n\ge 1\}\in\Gamma$ be arbitrary.
By the concavity of the function $x\mapsto \sqrt{x}\wedge 1$, \eqref{I},   Doob's Maximal Inequality,  and \eqref{L2limit}
\begin{equation}
\begin{aligned}
&\ \lim_{n\to \infty}\sum_{m=1}^\infty {1\over 2^m} E^z\left[\sup_{t\le m}\left\{\|I_{\delta(n)}(\omega)(t)-W(t,\omega)\|\wedge 1\right\}\right]\\
&\le \lim_{n\to \infty}\sum_{m=1}^\infty {1\over 2^m}\left\{\left( E^z\left[\sup_{t\le m}\left\{\|I_{\delta(n)}(\omega)(t)-W(t,\omega)\|^2\right\}\right]\right)^{1/2}\wedge 1\right\}\\
&\le \lim_{n\to \infty}\sum_{m=1}^\infty {1\over 2^m}\left\{\left( E^z\left[\sup_{t\le m}\left\{\|W^{\delta(n)}(t,\omega)-W(t,\omega)\|^2\right\}\right]\right)^{1/2}\wedge 1\right\}\\
&\le \lim_{n\to \infty}\sum_{m=1}^\infty {1\over 2^m}\left\{\left( 4\sum_{i=1,2}E^z\left[\left(W_i^{\delta(n)}(m)-W_i(m)\right)^2\right]\right)^{1/2}\wedge 1\right\}\\
&=0.\label{long}
\end{aligned}
\end{equation}
Then there exists a subsequence $\{\gamma^z(n),n\ge 1\}\subset \{\delta(n),n\ge 1\}$ such that
$$\lim_{n\to 0}\sum_{m=1}^\infty {1\over 2^m} \sup_{t\le m}\left\{\|I_{\gamma^z(n)}(\omega)(t)-W(t,\omega)\|\wedge 1\right\}= 0,$$
 $P^z$-a.s., which implies \eqref{gamma:z}. Next we are going to show the statement in (a). From our assumption that $P^z(F_\gamma(\omega)=\underline 0)=0$, it follows that $\lim_{n\to\infty}I_{\gamma(n)}(\omega)= F_\gamma(\omega)$ in the topology of uniform convergence on compacts, $P^z$-a.s.
On the other hand, similarly to \eqref{long} one can show that
$$\lim_{n\to 0}\sum_{m=1}^\infty {1\over 2^m} E^z\left[\sup_{t\le m}\left\{\|I_{\gamma(n)}(\omega)(t)-W(t,\omega)\|\wedge 1\right\}\right]=0,$$
which implies \eqref{more}.
 Next we are going to show part (b). By \eqref{F:shifted} and \eqref{gamma:z} we have $P^z$-a.s.
\begin{equation}F_{\gamma^z}(\omega(\tau+\cdot))
=F_{\gamma^z}(\omega)(\tau+\cdot)-F_{\gamma^z}(\omega)(\tau)=W(\tau+\cdot)-W(\tau)\not=\underline 0,\label{notzero}\end{equation}
thus by the strong Markov property
$$1=P^z\left(F_{\gamma^z}(\omega(\tau+\cdot))\not=\underline 0\right)=
P^z\left(F_{\gamma^z}(Z(\tau+\cdot))\not=\underline 0\right)=
E^z\left[P^{Z(\tau)}\left(F_{\gamma^z}(Z(\cdot))\not=\underline 0\right)\right].$$
Therefore there exists a set $S^\tau\in {\mathcal B}_S$ such that $P^z(Z(\tau)\in S^\tau)=1$ and
$$P^{z'}\left(F_{\gamma^z}(Z(\cdot))\not=\underline 0\right)=1 \ \hbox{for all}\ z'\in S^\tau$$
 ($S^\tau$ is a Borel set by Varadhan and Williams \cite{varadhan1985brownian}, Corollary 3.3). Then by part (a) we have $P^{z'}(F_{\gamma^z}(\omega)=W(\cdot,\omega))=1$ for all $z'\in S^\tau$, which implies \eqref{stoppingtime}.  Finally we show part (c). By part (a) all we need to show that $P^0(F_{\gamma^z}(\omega)=\underline 0)=0$. By the strong Markov property the left-hand side is equal to $P^z(F_{\gamma^z}(Z(\tau_0+\cdot))=\underline 0)$, which is indeed 0 by \eqref{notzero}.
\end{proof}

For every $z\in S$, we fix a sequence $\{\gamma^z(n),n\ge 1\}$ such that \eqref{gamma:z} holds. Recall the class ${\mathcal A}$ from Definition \ref{Definition:Of:The:Set:A}.

\begin{proposition}\label{p3} For each $z\in S$, we have $P^z$-a.s. that $\left(Z(\cdot\wedge \tau_0),Y(\cdot\wedge\tau_0)\right)\in {\mathcal A}$.
\end{proposition}

\begin{proof} \eqref{gamma:z} implies that for any subsequence $\{\alpha(n), n\ge 1\}\subset \{\gamma^z(n), n\ge 1\}$ we have $F_\alpha(\omega)=W(\cdot,\omega)$, $P^z$-a.s., and by \eqref{stopping} and by $Z(\cdot,\omega)=\omega$ we have
$$F_\alpha(Z(\cdot\wedge\tau_0))=W(\cdot\wedge\tau_0)=X(\cdot\wedge\tau_0)-z,$$
$P^z$-a.s. Then it is sufficient to show that there exists a sequence $\{\alpha(n),n\ge 1\}\subset \{\gamma^z(n),n\ge 1\}$ such that
\begin{equation}F_\alpha(\tilde Z(\cdot))=\tilde X(\cdot\wedge\tilde\tau_0)-z,\ Q^z\hbox{-a.s.}\label{sufficient}\end{equation}
Indeed, suppose that \eqref{sufficient} holds, then
\begin{eqnarray*} (\tilde Z(\cdot),\tilde Y(\cdot))=(\tilde Z(\cdot), \tilde Z(\cdot)-\tilde X(\cdot\wedge\tilde\tau_0))&=&
(\tilde Z(\cdot), \tilde Z(\cdot)-F_\alpha(\tilde Z(\cdot))-z)
\buildrel \rm d\over =\\
( Z(\cdot\wedge\tau_0),  Z(\cdot\wedge\tau_0)-F_\alpha(Z(\cdot\wedge\tau_0))-z) &=&
( Z(\cdot\wedge\tau_0), Y(\cdot\wedge\tau_0)),
\end{eqnarray*}
where the sign $\buildrel\rm d\over =$ means that the law of the expression on the left-hand side under $Q^z$ agrees with the law of the expression on the right-hand side under $P^z$. Then the statement of the present proposition follows from Proposition \ref{williams}. For the rest of this proof we shall prove \eqref{sufficient} which amounts to showing that for some $\{\alpha(n),n\ge 1\}\subset \{\gamma^z(n),n\ge 1\}$ and every $m\in{\mathbb N}_+$ we have
\begin{equation}\sup_{t\le m}\left\|I_{\alpha(n)}(\tilde Z(\cdot))(t) -\left(\tilde X(t\wedge\tilde\tau_0)-z\right)\right\|^2\to 0\label{also:sufficient}\end{equation}
as $n\to\infty$, $Q^z$-a.s.

Analog to $\tau_0^\delta, \tau_1^\delta,\sigma_1^\delta,\tau_2^\delta,\sigma_2^\delta \dots$ for $\delta>0$ we define the stopping times  $\tilde\tau_0^\delta=0$,
$\tilde \sigma_k^\delta=\inf\{t\ge\tilde \tau_{k-1}^\delta: \tilde Z(t)\in S_{2\delta}\}$ and $\tilde\tau_k^\delta=\inf\{t\ge \tilde\sigma_k^\delta: \tilde Z(t)\in \partial S_{\delta}\}, k\ge 1$. In the rest of the proof  we shall drop the superscript $z$ from $\gamma^z(n)$ whenever it leads to double superscript, and   write $\tilde\sigma^{\gamma(n)}$ and $\tilde\tau^{\gamma(n)}$ instead.

By the fact that $\tilde Z$ is flat on $[\tilde \tau_0,\infty)$, we have
\begin{eqnarray*}
I_{\gamma^z(n)}(\tilde Z(\cdot))(t\wedge m)&=&\sum_{k=1}^\infty\left[\tilde Z(t\wedge\tilde\tau_0\wedge\tilde\tau_k^{\gamma(n)}\wedge m) - \tilde Z(t\wedge\tilde\tau_0 \wedge\tilde\sigma_k^{\gamma(n)}\wedge m) \right].
\end{eqnarray*}
Moreover, since $\tilde Y$ is flat on $[\tilde\sigma_k^\delta\wedge m,\tilde\tau_k^\delta\wedge m]$,  we can cast this in the form
\begin{eqnarray*}
I_{\gamma^z(n)}(\tilde Z(\cdot))(t\wedge m)&=&\sum_{k=1}^\infty\left[\tilde X(t\wedge \tilde\tau_0\wedge\tilde\tau_k^{\gamma(n)}\wedge m) - \tilde X(t\wedge\tilde\tau_0\wedge\tilde\sigma_k^{\gamma(n)}\wedge m) \right] \nonumber \nonumber \\
&=&\int_{0}^t \sum_{k=1}^\infty 1\{u\in[ \tilde\sigma_k^{\gamma(n)}\wedge\tilde\tau_0\wedge m,  \tilde\tau_k^{\gamma(n)}\wedge\tilde\tau_0\wedge m]\}  d\tilde X(u).
\end{eqnarray*}
Hence by the the Dominated Convergence Theorem, the fact that $Z(\cdot\wedge\tau_0)$ has the same law as $\tilde Z(\cdot\wedge\tilde\tau_0)$, and by Williams \cite{williams1985recurrence}, Lemma 4.2, it follows that
\begin{eqnarray*}
&&\lim_{n\to\infty} E^z_Q\left[ \left\|I_{\gamma(n)}(\tilde Z(\cdot))(\tilde\tau_0\wedge m)-\left(\tilde X(\tilde\tau_0\wedge m)-z\right)\right\|^2\right] \nonumber\\
&=& 2\lim_{n\to\infty}E_Q^z\left[\sum_{k=0}^\infty \int_0^{\tilde\tau_0\wedge m} 1\{u\in [\tilde\tau_k^{\gamma(n)}, \tilde\sigma_{k+1}^{\gamma(n)}] \}du\right]  \nonumber\\
&\leq& 2\lim_{n\to\infty}E_Q^z\left[ \int_0^{\tilde\tau_0\wedge m} 1\{\tilde Z(u)\in S\setminus S_{2\gamma^z(n)}\}du\right]  \nonumber\\
&=& 2E_Q^z \left[\int_0^{\tilde\tau_0\wedge m} 1\{\tilde Z(u)\in\partial S\}du \right]  \nonumber\\
& =&
2 E^z \left[\int_0^{\tau_0\wedge m} 1\{ Z(u)\in\partial S\}du \right] \nonumber \\
&=&0.
\end{eqnarray*}
Then Doob's inequality applied to the square-integrable martingale $I_{\gamma(n)}(\tilde Z\wedge m)(\cdot)-(\tilde X(\cdot\wedge\tilde\tau_0\wedge m)-z)$ gives
\begin{eqnarray*}
&&E_Q^z\left[\sup_{t\le m}\left\{\left\|I_{\gamma(n)}(\tilde Z(\cdot))(t\wedge\tilde\tau_0)-\left(\tilde X(t\wedge\tilde\tau_0)-z\right)\right\|^2\right\}\right]\\
&\leq&4E_Q^z\left[\left\|I_{\gamma(n)}(\tilde Z(\cdot))(\tilde\tau_0\wedge m )-\left(\tilde X(\tilde\tau_0\wedge m)-z\right)\right\|^2\right]\to 0, \hbox{ as }n\to\infty,
\end{eqnarray*}
and so the existence of a subsequence
$\{\alpha(n),n\geq 1\}\subset\{\gamma^z(n),n\geq 1\}$ follows such that \eqref{also:sufficient} holds.
\end{proof}

\begin{proposition}\label{main}
For each $z\in S$ , and each  excursion interval $[G_i,D_i]$ of $Z$ away from the origin, we have $P^z$-a.s. that
\begin{equation*} \left(Z\left(\left(G_i+\cdot\right)\wedge D_i\right),  Y\left(\left(G_i+\cdot\right)\wedge D_i\right)\right)\in{\mathcal A}.\end{equation*}
\end{proposition}

\begin{proof}

Let $\{(U^n_k,T^n_k), k\ge 1\}$ be the sequence of excursion intervals of $Z$ away from zero with length strictly larger than $1/n$, i.e., $T^n_k-U^n_k>1/n$ for every $n\in{\mathbb N}_+$, and let $V_k^n=U_k^n+1/n$.
 Clearly it is sufficient to show that for all $k,n\in{\mathbb N}_+$ and $z \in S$,
\begin{equation*}P^z\left(\left(Z\left(\left(V_k^n+\cdot\right)\wedge T_k^n\right), Y\left(\left(V_k^n+\cdot\right)\wedge T_k^n\right)\right)\in {\mathcal A}\right)=1,
\end{equation*}
which is the same as
\begin{equation}P^z\left(\left(Z\left(\left(V_k^n+\cdot\right)\wedge T_k^n\right), Z\left(\left(V_k^n+\cdot\right)\wedge T_k^n\right)
-X\left(\left(V_k^n+\cdot\right)\wedge T_k^n\right)
\right)\in {\mathcal A}\right)=1.\label{new}\end{equation}
By \eqref{gamma:z} and \eqref{shiftedandstopped} we have $P^z$-a.s. that
$$X\left(\left(V_k^n+t\right)\wedge T_k^n\right)= F_{\gamma^z}(Z(\cdot)) \left(\left(V_k^n+t\right)\wedge T_k^n\right)+z =
F_{\gamma^z} \left(Z \left(\left(V_k^n+\cdot\right)\wedge T_k^n\right)\right)(t) +F_{\gamma^z}\left(Z(\cdot)\right)(V_k^n)+z,$$
hence we can replace the $X\left(\left(V_k^n+\cdot\right)\wedge T_k^n\right)$ term in \eqref{new} by $F_{\gamma^z} \left(Z \left(\left(V_k^n+\cdot\right)\wedge T_k^n\right)\right)$, since the terms $F_{\gamma^z}\left(Z(\cdot)\right)(V_k^n)+z$ have no impact on the membership in ${\mathcal A}$. Thus  the left-hand side of \eqref{new} is equal to
\begin{equation*}P^z\left(\left(Z\left(\left(V_k^n+\cdot\right)\wedge T_k^n\right), Z\left(\left(V_k^n+\cdot\right)\wedge T_k^n\right)- F_{\gamma^z} \left(Z \left(\left(V_k^n+\cdot\right)\wedge T_k^n\right)\right)\in {\mathcal A}\right)\right),\end{equation*}
and by the strong Markov property and \eqref{stopping} this is equal to
$$E^z\left[P^{Z(V_k^n)}\left(\left(Z(\cdot\wedge\tau_0), Z(\cdot\wedge\tau_0)- F_{\gamma^z}\left(Z(\cdot\wedge\tau_0)\right)\right)\in {\mathcal A}\right)\right]$$
\begin{equation}=E^z\left[P^{Z(V_k^n)}\left(\left(Z(\cdot\wedge\tau_0), Z(\cdot\wedge\tau_0)- F_{\gamma^z}\left(Z(\cdot)\right)(t\wedge\tau_0)\right)\in {\mathcal A}\right)\right].\label{step}\end{equation}
When applying \eqref{stopping} above, we needed the condition $E^z\left[P^{Z(V_k^n)}\left(F_{\gamma^z}\left(Z(\cdot)\right)\not=\underline 0\right)\right]$, but this follows from \eqref{stoppingtime}. Applying  \eqref{stoppingtime} and Proposition \ref{p3}, we get that \eqref{step} is equal to
 $$E^z\left[P^{Z(V_k^n)}\left(\left(Z(\cdot\wedge\tau_0), Z(\cdot\wedge\tau_0)-X(\cdot\wedge\tau_0)\right)\in {\mathcal A}\right)\right]=1,$$
as desired.
\end{proof}

Now we have the following lemma.

\begin{lemma}Suppose that $(z,y) \in \mathcal{A}$ and $T > 0$ are such that $ z(s) \neq 0$ for $s \in (0,T)$.
Then, letting $u=(u_1,u_2)$ be the non-decreasing function from Definition \ref{Definition:Of:The:Set:A}, we have that for $t\in [0,T]$,
\begin{eqnarray}
V_1(y,[0,t]) &=& \| v_1 \|(u_1(t)-u_1(0)) + \|v_2\|(u_2(t)-u_2(0)). \label{display:variation:of:y:over:excursion}
\end{eqnarray}
\end{lemma}

\begin{proof} The main idea of this proof is that on a closed interval included in $(0,T)$ the process $Z$ can hit the boundaries $\partial S_1$ and $\partial S_2$ only finitely many times. This may not be true on the interval $[0,T]$, because $z(0)$ or $z(T)$ may possibly be equal to zero. In this spirit let $0 < r < t < T$. We first show that
\begin{eqnarray}
V_1(y,[r,t]) &=& \| v_1 \|(u_1(t)-u_1(r)) + \|v_2\|(u_2(t)-u_2(r)). \label{display:to:use:triangle:inequality}
\end{eqnarray}
A straightforward application of the triangle inequality guarantees that the left-hand side of \eqref{display:to:use:triangle:inequality} is dominated by the right-hand side. From the continuity of $z$ and by the assumption $z(s) \neq 0$ for $s \in [r,t]$, it follows that there exits a finite partition $\pi = \{r=t_0 < t_1 < ... < t_n=t\}$ of $[r,t]$ such that for each $i=0,...,n-1$, within the interval $[t_i,t_{i+1}]$, $z$ does not touch both boundaries $\partial S_1$ and $\partial S_2$. We can assume without loss of generality that $z$ does not reach $\partial S_1$ in $[t_i,t_{i+1}]$ if $i$ is odd, and
does not reach $\partial S_2$ in $[t_i,t_{i+1}]$ if $i$ is even. Then
\begin{eqnarray*}
\sum_{i=0}^{n-1}\|y(t_{i+1})-y(t_i)\| &=&\sum_{i=0}^{n-1} \| v_1 u_1(t_{i+1})-v_1u_1(t_i) + v_2u_2(t_{i+1})-v_2u_2(t_i)\| \\
&=&\sum_{i~ \textrm{even}} \| v_1 u_1(t_{i+1})-v_1u_1(t_i)\| +\sum_{i ~ \textrm{odd}} \|  v_2 u_2(t_{i+1})-v_2u_2(t_i)\| \\
&=&\|v_1\|(u_1(t)-u_1(r))+\|v_2\|(u_2(t)-u_2(r)),
\end{eqnarray*}
which indeed proves \eqref{display:to:use:triangle:inequality}.

Now, from statement P6 in Section 2.2 of Chistyakov and Galkin \cite{chistyakov1998maps}, it follows that
\begin{eqnarray*}
V_1(y,(0,t]) &=& \| v_1 \|(u_1(t)-u_1(0)) + \|v_2\|(u_2(t)-u_2(0)), \end{eqnarray*}
for $0 < t < T$, and
\begin{eqnarray*}
V_1(y,(0,T)) &=& \| v_1 \|(u_1(T)-u_1(0)) + \|v_2\|(u_2(T)-u_2(0)). \end{eqnarray*}
However, based on the continuity of $y$ one can easily show that
$V_1(y,[0,t])=V_1(y,(0,t])$ and $V_1(y,[0,T])=V_1(y,(0,T))$, and hence \eqref{display:variation:of:y:over:excursion} follows.
\end{proof}

We now present the proof of Part \ref{Proposition:Skorohod:Problem:Part:2} of  Theorem \ref{Proposition:Skorokhod:Problem}.

\begin{proof}[Proof of Part \ref{Proposition:Skorohod:Problem:Part:2} of Theorem \ref{Proposition:Skorokhod:Problem}:]
Let $[G_i,D_i]$ be an arbitrary excursion interval of $Z(\cdot, \omega)$
away from zero. By Proposition \ref{main} there exist two non-decreasing functions $U_1,U_2: \mathbb{R}_+ \mapsto \mathbb{R}$, such that $Y((G_i+t) \wedge D_i)=v_1 U_1(t) + v_2 U_2(t)$ for $t \geq 0$, and
\begin{eqnarray*}
\int_0^{\infty}1\{Z((G_i+p)\wedge D_i) \notin \partial S_j\}dU_j(p) &=& 0,
\end{eqnarray*}
for $j=1,2$. Then, by the previous lemma,
\begin{eqnarray}
V_1(Y((G_i+\cdot) \wedge D_i),[0,t])&=& \| v_1 \|(U_1(t)-U_1(0)) + \|v_2\|(U_2(t)-U_2(0)),~~t \geq 0. \label{display:to:be:used:with:item:4}
\end{eqnarray}
The first three requirements of Definition \ref{Definition:Skorohod:Problem} are obviously satisfied. Item \ref{Definition:SP:Item:Four}
of Definition \ref{Definition:Skorohod:Problem} follows from \eqref{display:to:be:used:with:item:4}. Moreover, from \eqref{display:to:be:used:with:item:4} it follows that
\begin{eqnarray*}
Y((G_i+t) \wedge D_i) -Y(G_i)&=&\int_0^t \gamma(p)d( \| v_1 \|U_1(p) +\|v_2\|U_2(p)),
\end{eqnarray*}
where
\begin{eqnarray*}
\gamma(p) &=& \frac{v_1}{\|v_1\|}1\{Z((G_i+p)\wedge D_i) \in \partial S_1\}
+ \frac{v_2}{\|v_2\|}1\{Z((G_i+p)\wedge D_i) \in \partial S_2\},
\end{eqnarray*}
and so Item \ref{Definition:SP:Item:Five} of Definition \ref{Definition:Skorohod:Problem} follows. Clearly the same proof applies if instead of $[G_i,D_i]$ we consider the interval $[0,\tau_0]$, whenever $\tau_0>0$; the only difference is that we have to use Proposition \ref{p3} instead of Proposition \ref{main}.
\end{proof}

\section{Proof of Theorem \ref{Theorem:Main:Dirichlet:Process:Result}}
\label{section:dirichlet:process}

Let the process $X$ be as in Theorem \ref{Proposition:Skorokhod:Problem}, so
 for each $z \in S$ the process $X$ is a Brownian motion started from $z$ on $(C_S,\mathcal{F},\mathcal{F}_t,P^z)$, and let $Y$ be defined by $Z=X+Y$. In particular, $X$ is a local martingale on $(C_S,\mathcal{F},\mathcal{F}_t,P^z)$. Hence, by Definition \ref{Definition:Dirichlet:Process} of Section \ref{Section:Introduction}, in order to prove Theorem \ref{Theorem:Main:Dirichlet:Process:Result}
it suffices to prove that for each $z \in S$, $Y$ is a continuous, zero-energy process (see Definition \ref{Definition:Zero:Energy}) on $(C_S,\mathcal{F},\mathcal{F}_t,P^z)$ with $P^z(Y(0)=0)=1$. The
fact that $Y$ is continuous with $P^z(Y(0)=0)=1$ is immediate since $Y=Z-X$,
and so the proof of Theorem \ref{Theorem:Main:Dirichlet:Process:Result} is reduced to proving that $Y$ is of zero-energy on $(C_S,\mathcal{F},\mathcal{F}_t,P^z)$.

In order to prove that $Y$ is of zero-energy on $(C_S,\mathcal{F},\mathcal{F}_t,P^z)$ for each $z \in S$, we will rely on the relationship between functions of finite $p$-variation and $1/p$-H\"{o}lder continuous functions. In particular, let $E \subseteq \mathbb{R}_{+}$ and $f:E \mapsto \mathbb{R}^d$. Then, for each $\gamma > 0$ the function $f$ is said to be H\"{o}lder continuous
with exponent $\gamma$ and H\"{o}lder constant $H(f) \in \mathbb{R}_{+}$, if
\begin{eqnarray*}
\| f(t)-f(s)\| &\leq H(f) |t-s |^{\gamma}&~\textrm{for}~s,t \in E.
\end{eqnarray*}
In addition, we say that $f$ is locally H\"{o}lder continuous on $\mathbb{R}_{+}$ if it is H\"{o}lder continuous on compact $E \subset \mathbb{R}_{+}$. The following result now relates functions of finite $p$-variation to H\"{o}lder continuous functions with exponent $1/p$, see for instance Theorem 3.1 of Chistyakov and Galkin \cite{chistyakov1998maps} where it is proven for $p > 1$. The fact that it is true for $p > 0$ may be verified by mimicking their proofs.

\begin{proposition}\label{Proposition:Finite:Variation:Holder}
Let $E \subseteq \mathbb{R}_{+}$ and $f:E \mapsto \mathbb{R}^d$. Then for each $p > 0$, we have that $V_p(f,E) < \infty$
if and only if there exists a bounded, nondecreasing function $\varphi : E \mapsto \mathbb{R}$ and a map $g:\varphi(E) \mapsto \mathbb{R}^d$
which is H\"{o}lder continuous with exponent $1/p$ and H\"{o}lder constant $H(g) \leq 1$, such that $f = g \circ \varphi$ on $E$.
\end{proposition}

\noindent In the above, the quantity $\varphi(E) \subseteq \mathbb{R}$ is the range of the function $\varphi : E \mapsto \mathbb{R}$.  The requirement that the H\"{o}lder constant $H(g) \leq 1$ is unnecessary in the if portion of Proposition \ref{Proposition:Finite:Variation:Holder}. Indeed, $V_p(f,E) < \infty$ if $f = g \circ \varphi$ on $E$ for a bounded, nondecreasing function $\varphi : E \mapsto \mathbb{R}$ and any $g:\varphi(E) \mapsto \mathbb{R}^d$
which is H\"{o}lder continuous with exponent $1/p$, regardless of its constant. We also note that given a $f$ with $V_p(f,E) < \infty$, one particular choice of $\varphi$ and $g$ satisfying the properties of Proposition \ref{Proposition:Finite:Variation:Holder} is as follows (see the proof of Theorem 3.1 of \cite{chistyakov1998maps} for details). For each $t \in E$, let $E_t^{-}=\{s \in E : s \leq t\}$. Also, for future reference, set $E_t^{+}=\{s \in E : t \leq s\}$, and for $a,b \in E$ with $a \leq b$, let $E^b_a=E^{+}_a \cap E^{-}_b$. Now set $\varphi(t)=V_p(f,E_t^{-})$ for $t \in E$. Next, for $\tau \in \varphi(E)$, define the multi-valued inverse function $\varphi^{-1}(\tau)=\{t \in E : \varphi(t)=\tau\}$. We may then define $g: \varphi(E) \mapsto \mathbb{R}$ by $g(\tau)=f(t)$ for any point $t \in \varphi^{-1}(\tau)$. The fact that $g$ is well-defined is proven in \cite{chistyakov1998maps}. Finally, we note that
in the event that $E=[0,T]$ for some $T > 0$ and the function $f$ is continuous, then we may take the function $\varphi$ to be continuous as well (see Proposition 5.14 of Friz and Victoir \cite{friz2010multidimensional}).

\begin{proposition}\label{Proposition:Variation:Time:Chanaged:Holder}
Suppose that $\alpha < p < 2$. Then, for $P^0$-a.e. $\omega \in C_S$, there exists a continuous, nondecreasing function $\varphi_p : \mathbb{R}_{+} \mapsto \mathbb{R}$ and a map $g_p:\varphi_p(E) \mapsto \mathbb{R}^2$, which is locally H\"{o}lder continuous with exponent $1/p$, such that $Y(\cdot,\omega) = g_p \circ \varphi_p$ on $\mathbb{R}_{+}$.
\end{proposition}

In the following two subsections, we provide the proof of Proposition \ref{Proposition:Variation:Time:Chanaged:Holder}. We close this subsection with the proof of Theorem \ref{Theorem:Main:Dirichlet:Process:Result} which follows as a result of Propositions \ref{Proposition:Finite:Variation:Holder} and \ref{Proposition:Variation:Time:Chanaged:Holder}, together with the results of Section \ref{Section:Identification:of:X:and:Skorokhod:Problem}.

\begin{proof}[Proof of Theorem \ref{Theorem:Main:Dirichlet:Process:Result}]We first treat the case of $z=0$. By the decomposition $Z=X+Y$, where $X$ is a Brownian motion started from $0$ on $(C_S,\mathcal{F},\mathcal{F}_t,P^0)$, in order to show that $Z$ is a Dirichlet process on $(C_S,\mathcal{F},\mathcal{F}_t,P^0)$, it suffices by Definition \ref{Definition:Dirichlet:Process} to prove that $Y$ is a process of zero energy on $(C_S,\mathcal{F},\mathcal{F}_t,P^0)$. Let $\alpha < p < 2$ and recall by Proposition \ref{Proposition:Variation:Time:Chanaged:Holder} that for $P^0$-a.e. $\omega \in C_S$, $Y(\cdot,\omega) = g_p \circ \varphi_p$ on $\mathbb{R}_{+}$, where $\varphi_p : \mathbb{R}_{+} \mapsto \mathbb{R}$ is a  continuous, nondecreasing function  and $g_p:\varphi_p(E) \mapsto \mathbb{R}^2$ is locally H\"{o}lder continuous on $\mathbb{R}_{+}$ with exponent $1/p$.

Now let $T \geq 0$ and let $H_{\varphi_p(T)}(g_p)$
denote the H\"{o}lder constant of $g_p$ on $[0,\varphi_p(T)]$. Then, for any partition $\pi$ of $[0,T]$,
\begin{eqnarray*}
\sum_{t_i \in \pi} \|Y(t_i)-Y(t_{i-1})  \|^2 &=&\sum_{t_i \in \pi} \|g_p \circ \varphi_p(t_i)-g_p \circ \varphi_p(t_{i-1})  \|^2 \\
&\leq& (H_{\varphi_p(T)}(g_p))^2\sum_{t_i \in \pi} |\varphi_p(t_i)-\varphi_p(t_{i-1})  |^{2/p} \\
&\leq&(H_{\varphi_p(T)}(g_p))^2(\varphi_p(T)-\varphi_p(0)) \max_{t_i \in \pi} |\varphi_p(t_i)-\varphi_p(t_{i-1})  |^{2/p-1}.
\end{eqnarray*}
Next let $\{\pi^n, n \geq 1\}$ be a sequence of partitions of $[0,T]$ with $\|\pi^n \| \rightarrow 0$ as $n \rightarrow \infty$, and note that since $ \alpha < p < 2$, $2/p-1 > 0$. Then, since $\varphi_p$
is $P^0$-a.s. continuous, it follows that $\max_{t_i \in \pi^n} |\varphi_p(t_i)-\varphi_p(t_{i-1})  |^{2/p-1} \rightarrow 0$ as $n \rightarrow \infty,~P^0$-a.s., and so by the above
\begin{eqnarray*}
\sum_{t_i \in \pi} \|Y(t_i)-Y(t_{i-1})  \|^2 &\rightarrow&0~~\textrm{as}~~n \rightarrow \infty,~P^0\textrm{-a.s.},
\end{eqnarray*}
which implies that $Y$ is of zero-energy on $(C_S,\mathcal{F},\mathcal{F}_t,P^0)$. Thus, $Z$ is a Dirichlet process on $(C_S,\mathcal{F},\mathcal{F}_t,P^0)$.

Now let $z \in S$ arbitrary. Again by the decomposition $Z=X+Y$, where $X$ is a Brownian motion started from $z$ on $(C_S,\mathcal{F},\mathcal{F}_t,P^z)$ \ it  suffices by Definition \ref{Definition:Dirichlet:Process} to prove that $Y$ is a process of zero energy on $(C_S,\mathcal{F},\mathcal{F}_t,P^z)$. Using Proposition \ref{Proposition:Finite:Variation:Holder} above and similar arguments to the case of $z=0$ in the preceding paragraphs, it suffices to show that for $1 \leq \alpha < p < 2$,
\begin{eqnarray}
P^z(V_p(Y,[0,T]) < \infty)&=&1,~T \geq 0. \label{display:used:before:in:proof}
\end{eqnarray}

Recall from \eqref{tau:zero} of Section \ref{Section:Identification:of:X:and:Skorokhod:Problem} the definition of the stopping time $\tau_0=\inf\{t > 0 :Z(t)=0\}$. By Theorem 2.2 of Varadhan amd Williams \cite{varadhan1985brownian}, since $ \alpha > 0$,  $P^z(\tau_0 < \infty)=1$.
By the monotonicity property (see (P2) in Chistyakov and Galkin \cite{chistyakov1998maps}),
$$P^z\left(V_p\left(Y,[0,T]\right)<\infty\right)\ge P^z\left(V_p\left(Y,[0,\tau_0+T]\right)<\infty\right).$$
By the  semi-additivity property (see  (P3) in Chistyakov and Galkin \cite{chistyakov1998maps})
the right-hand side in the above equation is equal to
$$P^z\left(V_p\left(Y,[0,\tau_0]\right)<\infty \ \hbox{and}\ V_p\left(Y,[\tau_0,\tau_0+T]\right)<\infty\right).$$
Now using the strong Markov property \cite{varadhan1985brownian}, this expression can be written as
\begin{eqnarray}
&&E^z\left[1\{V_p\left(Y,[0,\tau_0]\right)<\infty\}P^z(V_p\left(Y,[\tau_0,\tau_0+T]\right)<\infty|{\cal F}_{\tau_0})\right]  \label{use:markov}\\
&=&P^z\left(V_p\left(Y,[0,\tau_0]\right)<\infty\right)P^0\left(V_p\left(Y,[0,T]\right)<\infty\right).
 \nonumber
\end{eqnarray}
However, it has already been shown earlier in the proof that the second factor in \eqref{use:markov} is equal to 1. By Part \ref{Proposition:Skorohod:Problem:Part:2} of Theorem \ref{Proposition:Skorokhod:Problem}, $P^z\left(V_1\left(Y,[0,\tau_0]\right)<\infty\right)=1$, and this implies that the first term in \eqref{use:markov} is also equal to 1.
\end{proof}

\subsection{Properties of the Function $\mathbf{\varphi_{p,q}}$}

For each $p > \alpha$ and $q > \alpha/2$, define the function $\varphi_{p,q} : \mathbb{R}_{+} \mapsto \mathbb{R}_{+}$ by setting
\begin{eqnarray}
\varphi_{p,q}(t) &=&\begin{cases}
V_{q}(L^{-1}, [0,L(t)] ), &\text{for $t \in  L^{-1}(\mathbb{R}_{+})$,}\\
V_{q}(L^{-1}, [0,L(t)) ), &\text{for $t \in  \Lambda \backslash L^{-1}(\mathbb{R}_{+})$,} \\
\varphi_{p,q}(G_i)+ \frac{(D_i-G_i)^q}{V_p(Y ,[G_i,D_i])} \cdot V_p(Y  ,[G_i,t]) , &\text{for $t \in  (G_i,D_i)$ for $i \geq 1$,} \\
\end{cases} \label{Display:Definition:Of:Varphi}
\end{eqnarray}
where we recall from Section \ref{Section:Excursion:Theory:Background} the definition of $\Lambda$ as the zero set of $Z$, and of $L$ as the local time of $Z$ at the origin. Our main result of this subsection is the following.

\begin{proposition}\label{Proposition:Main:Properties:Of:Varphi:P:Q}
For each $ p > \alpha $ and $1 > q > \alpha/2$, the function $\varphi_{p,q} : \mathbb{R}_{+} \mapsto \mathbb{R}_{+}$ is $P^0$-a.s. well-defined, non-decreasing and continuous with $\varphi_{p,q}(\mathbb{R}_{+})=\mathbb{R}_{+}$.
\end{proposition}

In preparation for the proof of Proposition \ref{Proposition:Main:Properties:Of:Varphi:P:Q}, we provide three lemmas.

\begin{lemma}\label{Lemma:Properties:Of:Var:Phi:P:Q:On:Lambda}
For each $p > \alpha$ and $q > \alpha/2$, the function $\varphi_{p,q} : \mathbb{R}_{+} \mapsto \mathbb{R}_{+}$ is $P^0$-a.s. well-defined and strictly increasing on the set $\Lambda$. Moreover, if in addition to the above assumptions for $p$ and $q$ it is also true that $q<1$, then $P^0$-a.s.,
\begin{equation} \varphi_{p,q}(t)=\sum_{i:G_i<t}(D_i-G_i)^q,~t \in \Lambda.\label{nice:rep}
\end{equation}
\end{lemma}

\begin{proof}Let $p > \alpha$ and $q > \alpha/2$. First note that by \eqref{display:complement:of:lambda}, it follows that $P^0$-a.s. the three cases provided in the definition of $\varphi_{p,q}$ are disjoint. Now recall from Theorem 2.6 of Williams \cite{williams1987local} that $L^{-1}$ is a stable subordinator of index $\alpha/2$ under $P^0$. Hence, by the results of Simon \cite{simon2004small} (see immediately below display (9) of \cite{simon2004small}), it follows that $P^0(V_q(L^{-1},[0,1]) < \infty)=1$. Using the scaling property \cite{williams1987local} under $P^0$,
$$L^{-1}(\lambda \cdot)~ \buildrel d \over =~ \lambda^{2/\alpha}L^{-1}(\cdot),~\lambda > 0, $$
it then follows that $P^0(V_q(L^{-1},[0,T]) < \infty)=1$ for each $T \geq 0$. Moreover, by the monotonicity property of strong $q$-variation (see (P2) of Chistyakov and Galkin \cite{chistyakov1998maps}), $V_q(L^{-1},[0,\cdot])$ is $P^0$-a.s. a non-decreasing function. Hence, $P^0(V_q(L^{-1},[0,T]) < \infty, T \geq 0)=1$.  It now follows from \eqref{Display:Definition:Of:Varphi} that $\varphi_{p,q}$ is $P^0$-a.s. well-defined on the set $\Lambda$.

Next, we show \eqref{nice:rep}. It follows from  \eqref{l:inverse} of Section \ref{Section:Excursion:Theory:Background} and (9) of Simon \cite{simon2004small} that for $\alpha/2<q<1$, $P^0$-a.s.,       \begin{equation}V_q(L^{-1},[0,a])=\sum_{i:L(G_i)\le a}(D_i-G_i)^q
\ \ \hbox{and}
\ \ V_q(L^{-1},[0,a))=\sum_{i:L(G_i)< a}(D_i-G_i)^q,~a \geq 0.\label{open}\end{equation}
\eqref{nice:rep} now follows from \eqref{Display:Definition:Of:Varphi} and \eqref{open}, combined with the identity \eqref{Display:Disjoint:Intervals}.

Finally, we show that $\varphi_{p,q}$ is $P^0$-a.s. strictly increasing on the set $\Lambda$. By Item \ref{Condition:3:Submartingale} of Definition \ref{vw}, it follows that $P^0$-a.s. for every $t,t'\in\Lambda$ such that $t<t'$, there exists some $s\in(t,t')$ such that $s\notin\Lambda$. But then by \eqref{display:complement:of:lambda}, $G_i\in[t,t')$ for some $i \geq 1$, and so \eqref{nice:rep} implies $\varphi_{p,q}(t)<\varphi_{p,q}(t')$.
\end{proof}

\begin{lemma}\label{Lemma:Y:Positive:Variation}
For each $i \geq 1$ and $\varepsilon > 0$,
\begin{eqnarray}
P^0(V_1(Y,[G_i,D_i \wedge (G_i+\varepsilon)]) > 0)&=&1.  \label{Display:Y:Variation:Positive:On:Excursion}
\end{eqnarray}
\end{lemma}

\begin{proof}First, we need a slight generalization of Definition 2.9.10 and Theorem 2.9.12 in Karatzas and Shreve \cite{KaratzasShreve}. Let $f:\mathbb{R}_{+} \mapsto \mathbb{R}$ be given. A time $t\ge 0$ will be called a {\it point of local maximum (minimum) from the right} if there exists a $\delta>0$ such that $f(s)\le f(t)$ ($f(s)\ge f(t))$ for every $s\in(t,t+\delta)$. A time $t\ge0$ will be called a {\it point of strict local maximum (minimum) from the right} if there exists a $\delta>0$ such that $f(s)< f(t)$ ($f(s)>f(t)$) for every $s\in(t,t+\delta)$. Now let $w$ be a standard (1-dimensional) Brownian motion on some probability space. Then, following the proof of Theorem 2.9.12 in \cite{KaratzasShreve}, for almost every sample path of $w$ all points of local maximum from the right are strict.
In a similar manner, it is also true that for almost every path $w$ all points of local minimum from the right are strict.

Now we are ready to prove \eqref{Display:Y:Variation:Positive:On:Excursion}. We shall use the same notation as in Williams \cite{williams1987local}. In particular, we adjoin a point $\partial$ to $C_S$, and define the excursion process $\{U_t, t \geq 0\}$ as the $(C_S,{\cal M})$-valued point process given for $t \geq 0$ by
\begin{eqnarray*}
U_t(s)&=&\begin{cases}
Z\left(L^{-1}(t-)+s\right), &\textrm{if}~~ 0\le s <L^{-1}(t)-L^{-1}(t-), \\
0,  &\textrm{if}~~  s\ge L^{-1}(t)-L^{-1}(t-)>0, \\
\partial, &\textrm{if}~~ L^{-1}(t)-L^{-1}(t-)=0.
\end{cases}
\end{eqnarray*}
Now let $N(A,U)$ be the number of times $t$ such that $(t,U_t)\in A$, where $A\in {\cal B}\times{\cal M}$ and ${\cal B}$ is the $\sigma$-field of Borel subsets of $(0,\infty)$. Then, from Theorem 3.3 of Williams \cite{williams1987local} it follows that
\begin{eqnarray*}
E^0\left[N\left(\left(0,t\right)\times\left\{\omega: T_{\partial S}(\omega)>0\right\},U\right)\right]&=&0,
\end{eqnarray*}
where $T_{\partial S}(\omega)=\inf\left\{t>0: \omega(t)\in\partial S\right\}$ for $\omega\in C_S$. Then, by the Monotone Convergence Theorem \cite{lieb2001analysis},
\begin{eqnarray*}
E^0\left[N\left(\left(0,\infty\right)\times\left\{\omega: T_{\partial S}(\omega)>0\right\},U\right)\right]&=&0.
\end{eqnarray*}
This implies that for $P^0$-a.e. sample path $T_{\partial S}^i=G_i$ for all $i\ge 1$, where
\begin{eqnarray*}
T_{\partial S}^i&=&\inf\left\{t>G_i: Z(t)\in\partial S\right\}.
\end{eqnarray*}
Suppose now that \eqref{Display:Y:Variation:Positive:On:Excursion} does not hold. That is, $Y$ is flat on $[G_i,(G_i+\epsilon)\wedge D_i]$ for some $i\ge 1$. Since $Z(G_i)=0$,
\begin{equation} Z(t)=X(t)+Y(G_i)=X(t)-X(G_i),\quad t\in[G_i,(G_i+\epsilon)\wedge D_i].\label{(2)}\end{equation}
Now $T_{\partial S}^i=G_i$ implies the existence of a sequence $t_n\downarrow G_i$ such that $Z(t_n)\in\partial S_1$ for all $n\ge 1$, or $Z(t_n)\in\partial S_2$ for all $n\ge 1$. In the first case, $Z_2(t_n)=0$ for all $n\ge 1$. Thus, by \eqref{(2)}, $G_i$ is a point of local minimum from the right for the Brownian motion $X_2$, but not strict, which is a contradiction. In the second case, we proceed similarly. $G_i$ is a point of local minimum from the right for the Brownian motion $\langle n_2, X\rangle$, but not  strict, which is again a contradiction.
\end{proof}

\begin{lemma}\label{Lemma:Properties:Of:Phi:P:Q:On:Lambda:C}
For each $ p > \alpha $ and $q > \alpha/2$, the function $\varphi_{p,q} : \mathbb{R}_{+} \mapsto \mathbb{R}_{+}$ is $P^0$-a.s. well-defined, non-decreasing and continuous on the set $(G_i,D_i)$ for $i \geq 1$.
\end{lemma}

\begin{proof}Let $p > \alpha$ and $q > \alpha/2$. By \eqref{display:complement:of:lambda}, it follows that $P^0$-a.s. the three cases provided in the definition of $\varphi_{p,q}$ are disjoint. We now show that $\varphi_{p,q}$ is $P^0$-a.s. well-defined on $(G_i,D_i)$ for each $i \geq 1$. By \eqref{Display:Definition:Of:Varphi} and the fact that $V_p(Y,[G_i,\cdot])$ is non-decreasing on $(G_i,D_i)$,  it suffices to show that $P^0$-a.s., $0 < V_p(Y,[G_i,D_i]) < +\infty$. The fact that $P^0$-a.s., $0 < V_p(Y,[G_i,D_i])$ follows by Lemma \ref{Lemma:Y:Positive:Variation} above. The fact that $P^0$-a.s., $V_p(Y,[G_i,D_i]) < +\infty$ follows since
by Part \ref{Proposition:Skorohod:Problem:Part:2} of Theorem \ref{Proposition:Skorokhod:Problem}, $V_1(Y,[G_i,D_i]) < +\infty$ $P^0$-a.s., and the fact that $V_p(Y,[G_i,D_i]) \leq (V_1(Y,[G_i,D_i]))^p  $ (see Remark 2.5 of Chistyakov and Galkin \cite{chistyakov1998maps}) since by assumption $p > \alpha > 1$. The fact that $\varphi_{p,q} : \mathbb{R}_{+} \mapsto \mathbb{R}_{+}$ is $P^0$-a.s. continuous and non-decreasing on $(G_i,D_i)$ follows by \eqref{Display:Definition:Of:Varphi} and the facts that $V_p(Y,[G_i,\cdot])$ is non-decreasing on $(G_i,D_i)$ and,
since $Y$ is $P^0$-a.s. continuous, $V_p(Y,[G_i,\cdot])$ is $P^0$-a.s. continuous on $(G_i,D_i)$ as well.
\end{proof}

Combining Lemmas \ref{Lemma:Properties:Of:Var:Phi:P:Q:On:Lambda} and \ref{Lemma:Properties:Of:Phi:P:Q:On:Lambda:C}, we now provide the proof of Proposition \ref{Proposition:Main:Properties:Of:Varphi:P:Q}.

\begin{proof}[Proof of Proposition \ref{Proposition:Main:Properties:Of:Varphi:P:Q}]Let $p > \alpha$ and $1 > q > \alpha/2$.  By \eqref{display:complement:of:lambda}, it follows that $P^0$-a.s. the three cases provided in the definition of $\varphi_{p,q}$ are disjoint and cover all of $\mathbb{R}_{+}$. The fact that $\varphi_{p,q} : \mathbb{R}_{+} \mapsto \mathbb{R}_{+}$ is $P^0$-a.s. well-defined and non-decreasing now follows immediately by Lemmas \ref{Lemma:Properties:Of:Var:Phi:P:Q:On:Lambda} and \ref{Lemma:Properties:Of:Phi:P:Q:On:Lambda:C}.

We next prove that $\varphi_{p,q} : \mathbb{R}_{+} \mapsto \mathbb{R}_{+}$ is $P^0$-a.s. continuous. By Lemma \ref{Lemma:Properties:Of:Phi:P:Q:On:Lambda:C}, it suffices to show that $P^0$-a.s. $\varphi_{p,q}$ is continuous at each point $t \in \Lambda$. Let $\mathcal{H}_1$ be the set of all $\omega \in C_S$ such that the following four conditions hold,
\begin{enumerate}
\item Identity \eqref{display:complement:of:lambda} holds,
\label{early:cond:display:complement}
\item Identity  \eqref{nice:rep} holds,
\label{early:cond:nice:rep}
\item $V_p(Y(\cdot,\omega),[G_i,\cdot])$ is continuous on $(G_i,D_i)$ for each $i \geq 1$,
\label{early:cond:continuous}
\item $\varphi_{p,q}(\cdot,\omega):\mathbb{R}_{+} \mapsto \mathbb{R}_{+}$ is well-defined and non-decreasing.
\label{early:cond:nondecreasing}
\end{enumerate}
By the discussion in Section \ref{Section:Excursion:Theory:Background}, together with the definition of $\varphi_{p,q}$ and Lemmas \ref{Lemma:Properties:Of:Var:Phi:P:Q:On:Lambda} and \ref{Lemma:Properties:Of:Phi:P:Q:On:Lambda:C}, and the discussion of the preceding paragraph, it follows that  $P^0(\mathcal{H}_1)$=1. It therefore
suffices to show that $\varphi_{p,q}(\cdot,\omega):\mathbb{R}_{+} \mapsto \mathbb{R}_{+} $ is  continuous
for each $\omega \in \mathcal{H}_1$. Hence, fix an arbitrary $\hat{\omega} \in \mathcal{H}_1$ for the remainder of this paragraph. Let $t \in \Lambda$ be arbitrary. We show that $\varphi_{p,q}$ is both left- and right-continuous at $t$. We start with left-continuity at $t$. Obviously then we can assume that $t>0$. If $t$ isolated from the left in $\Lambda$, then since identity \eqref{display:complement:of:lambda} holds by condition \ref{early:cond:display:complement}, $t=D_i$ for some $i \geq 1$, and so left-continuity follows from the fact that by condition \ref{early:cond:nondecreasing}, $\varphi_{p,q}$ is well-defined, the continuity of $V_p(Y,[G_i,\cdot])$ on $(G_i,D_i)$, which is guaranteed by condition \ref{early:cond:continuous}, and the identity \eqref{nice:rep} guaranteed by condition \ref{early:cond:nice:rep}. Suppose now that $t$ is not isolated from the left in $\Lambda$, i.e. there exists a sequence $\{t_n , n\ge 1\} \subset \Lambda$ such that $t_n\uparrow t$ as $n\to\infty$. Then, since \eqref{nice:rep} holds by condition \ref{early:cond:nice:rep}, it follows that
$$\lim_{n\to\infty}\left(\varphi_{p,q}(t) -  \varphi_{p,q}(t_n)\right) ~=~ \lim_{n\to\infty}\sum_{i:t_n\le G_i<t}(D_i-G_i)^q~=~0,$$
and so left-continuity follows since by condition \ref{early:cond:nondecreasing}, $\varphi_{p,q}$ is non-decreasing. The proof of right-continuity of $\varphi_{p,q}$ at $t$ is quite similar and so we omit the details.

We now complete the proof by showing that $P^0$-a.s., $\varphi_{p,q}(\mathbb{R}_{+})=\mathbb{R}_{+}$. First note that using the fact from Section \ref{Section:Excursion:Theory:Background} that  $L^{-1}$ is a subordinator under $P^0$, and the subadditivity property of strong $q$-variation (see (P3) in Section 2.2 of Chistyakov and Galkin \cite{chistyakov1998maps}), it is straightforward to show that $P^0$-a.s., $V_q(L^{-1},[0,A]) \rightarrow + \infty$ as $A \rightarrow +\infty$. Hence, since by the discussion in Section \ref{Section:Excursion:Theory:Background},  $P^0$-a.s., $L(t) \rightarrow + \infty$ as $t \rightarrow +\infty$, it follows by \eqref{Display:Definition:Of:Varphi} and the fact
that $\varphi_{p,q}:\mathbb{R}_{+} \mapsto \mathbb{R}_{+}$ is $P^0$-a.s. non-decreasing, that $P^0$-a.s., $\varphi_{p,q}(t) \rightarrow + \infty$ as $t \rightarrow + \infty$. Moreover, it is clear from \eqref{Display:Definition:Of:Varphi} that $P^0$-a.s., $\varphi_{p,q}(0)=0$. The desired result now follows since $\varphi_{p,q}$ is $P^0$-a.s. continuous on $\mathbb{R}_{+}$.
\end{proof}

\subsection{Proof of Proposition \ref{Proposition:Variation:Time:Chanaged:Holder}}

By Proposition \ref{Proposition:Main:Properties:Of:Varphi:P:Q}, we have that for each $p > \alpha$ and $1 > q  > \alpha/2$ the function $\varphi_{p,q}:\mathbb{R}_{+} \mapsto \mathbb{R}_{+}$ is $P^0$-a.s. well-defined, non-decreasing and continuous with $\varphi_{p,q}(\mathbb{R}_{+})=\mathbb{R}_{+}$. Hence, we may $P^0$-a.s. define its right-continuous inverse function
\begin{eqnarray*}
\varphi_{p,q}^{-1}(\tau) &=& \inf\{t \geq 0 : \varphi_{p,q}(t) > \tau\},~\tau \geq 0.
\end{eqnarray*}
In preparation for the proof of Proposition \ref{Proposition:Variation:Time:Chanaged:Holder}, we provide two lemmas regarding $\varphi_{p,q}^{-1}$.

\begin{lemma}\label{Lemma:Inverse:Of:Varphi:On:Zero:Set}
For each $p > \alpha$ and $1 > q > \alpha/2$, it follows that $P^0$-a.s.,  $\varphi_{p,q}^{-1} \circ \varphi_{p,q} (t)=t$ for $t \in \Lambda$. Moreover, $\varphi_{p,q}^{-1}$ is $P^0$-a.s. H\"{o}lder continuous on $\varphi_{p,q}(\Lambda)$ with H\"{o}lder exponent $1/q$ and H\"{o}lder constant $H(\varphi_{p,q}^{-1}) \leq 1$.
\end{lemma}

\begin{proof}Let $p > \alpha$ and $1 > q > \alpha/2$. Next, let $\mathcal{H}_2$ be the set of all $\omega \in C_S$ such that the following five conditions hold,
\begin{enumerate}
\item There are no isolated points in $\Lambda$ and the identities \eqref{display:complement:of:lambda} and \eqref{Display:Disjoint:Intervals} hold,
\label{middle:cond:lambda:set}

\item  $L^{-1}$ is strictly increasing and right-continuous,
\label{middle:cond:l:inverse}

\item $\varphi_{p,q}:\mathbb{R}_{+} \mapsto \mathbb{R}_{+}$ is well-defined,
non-decreasing and continuous with
$\varphi_{p,q}(\mathbb{R}_{+})=\mathbb{R}_{+}$,
\label{middle:cond:phi:on:r}

\item $\varphi_{p,q}$ is strictly increasing on $\Lambda$,
\label{middle:cond:phi:increasing}

\item $V_1(Y,[G_i,D_i \wedge (G_i+2^{-j})]) > 0$ for $i,j \geq 1$.
\label{middle:cond:phi:excursion:intervals}

\end{enumerate}
By the discussion in Section \ref{Section:Excursion:Theory:Background},  Proposition \ref{Proposition:Main:Properties:Of:Varphi:P:Q}, and Lemmas \ref{Lemma:Properties:Of:Var:Phi:P:Q:On:Lambda} and \ref{Lemma:Y:Positive:Variation}, it follows that $P^{0}(\mathcal{H}_2)=1$. It therefore suffices to show that the results of the lemma hold for each $\omega \in \mathcal{H}_2$. Hence, fix an arbitrary $\hat{\omega} \in \mathcal{H}_2$
for the remainder of the proof.

We first prove that $\varphi_{p,q}^{-1} \circ \varphi_{p,q} (t)=t$ for $t \in \Lambda$. Since the identities \eqref{display:complement:of:lambda} and \eqref{Display:Disjoint:Intervals} are guaranteed by condition \ref{middle:cond:lambda:set}, no point  $t \in  L^{-1}(\mathbb{R}_{+})$ is isolated from the right in $\Lambda$. Hence, since by condition \ref{middle:cond:phi:increasing},  $\varphi_{p,q}$ is strictly increasing on $\Lambda $, it follows that $\varphi_{p,q}^{-1} \circ \varphi_{p,q} (t)=t$ for $t \in L^{1}(\mathbb{R}_{+})$. Next, note that by \eqref{Display:Disjoint:Intervals}, which is guaranteed by condition \ref{middle:cond:lambda:set}, it follows that if $t \in \Lambda \backslash L^{1}(\mathbb{R}_{+})$, then $t=G_i$ for some $i \geq 1$.
Hence, since by condition \ref{middle:cond:phi:excursion:intervals}, $V_1(Y,[G_i,D_i \wedge (G_i+2^{-j})]) > 0$ for $i,j \geq 1$, it follows since by condition \ref{middle:cond:phi:on:r}, $\varphi_{p,q}$ in well-defined  and non-decreasing on $\mathbb{R}_{+}$, that using \eqref{Display:Definition:Of:Varphi}, $\varphi_{p,q}(t^{'}) > \varphi_{p,q}(t)$ for $t^{'} > t$. Hence, it follows that $\varphi_{p,q}^{-1} \circ \varphi_{p,q} (t)=t$ for $t \in \Lambda \backslash L^{1}(\mathbb{R}_{+})$, and so we have now shown that $\varphi_{p,q}^{-1} \circ \varphi_{p,q} (t)=t$ for $t \in \Lambda$, as desired.

Next, we show that $\varphi_{p,q}^{-1}$ is H\"{o}lder continuous on $\varphi_{p,q}(\Lambda)$ with H\"{o}lder exponent $1/q$ and H\"{o}lder constant $H(\varphi_{p,q}^{-1}) \leq 1$. First note that since by condition \ref{middle:cond:phi:on:r}, $\varphi_{p,q}$ is well-defined on $\mathbb{R}_{+}$, it follows using the definition of $\varphi_{p,q}$ in \eqref{Display:Definition:Of:Varphi} that
\begin{eqnarray*}
\varphi_{p,q}(t) &=& V_{q}(L^{-1},[0,L(t)]),~t \in L^{-1}(\mathbb{R}_{+}).
\end{eqnarray*}
Trivially, $L^{-1}=e \circ L^{-1}$, where $e:\mathbb{R}_{+} \mapsto \mathbb{R}_{+}$ is the identity function, and so since $L^{-1}$ is non-decreasing by condition \ref{middle:cond:l:inverse}, the change-of-variable formula for strong $q$-variation (see (P4) of Chistyakov and Galkin \cite{chistyakov1998maps}) implies that \begin{eqnarray*}
\varphi_{p,q}(t)~=~ V_{q}(L^{-1},[0,L(t)])&=& V_{q}(e,L^{-1}([0,L(t)]))~=~V_{q}(e,(L^{-1}(\mathbb{R}_{+}))_{t}^{-} ),~t \geq 0,
\end{eqnarray*}
where the final equality follows since by condition \ref{middle:cond:l:inverse}, $L^{-1}$ is non-decreasing, and since $L^{-1} \circ L(t)=t$ for $t \in L^{-1}(\mathbb{R}_{+})$. Now note that since by condition \ref{middle:cond:phi:on:r}, $\varphi_{p,q}$ is continuous on $\mathbb{R}_{+}$ with $\varphi_{p,q}(\mathbb{R}_{+})=\mathbb{R}_{+}$, it follows that $\varphi_{p,q} \circ \varphi^{-1}_{p,q}(\tau)=\tau$ for $\tau \in \mathbb{R}_{+}$. Moreover, since it has been shown in the preceding discussion that $\varphi_{p,q}^{-1} \circ \varphi_{p,q} (t)=t$ for $t \in \Lambda$, it follows that $\varphi^{-1}_{p,q}(\tau) \in L^{-1}(\mathbb{R}_{+})$ for $\tau \in \varphi_{p,q}(L^{-1}(\mathbb{R}_{+}))$. Thus,
for $\tau, \tau^{'} \in \varphi_{p,q}(L^{-1}(\mathbb{R}_{+}))$ with $\tau \leq \tau^{'}$, by the minimality property and the semi-additivity property of strong $q$-variation (see (P1) and (P3) of Chistyakov and Galkin \cite{chistyakov1998maps}), it follows that
\begin{eqnarray*}
|\varphi_{p,q}^{-1}(\tau^{'})-\varphi_{p,q}^{-1}(\tau) |^{q} &\leq&V_{q}\left(e,\left(L^{-1}(\mathbb{R}_{+})\right)_{\varphi_{p,q}^{-1}(\tau)}^{\varphi_{p,q}^{-1}(\tau^{'})} \right)\\
&\leq&V_{q}\left(e,\left(L^{-1}(\mathbb{R}_{+})\right)_{\varphi_{p,q}^{-1}(\tau^{'})}^{-} \right)-V_{q}\left(e,\left(L^{-1}(\mathbb{R}_{+})\right)_{\varphi_{p,q}^{-1}(\tau)}^{-} \right)\\
&=&\tau^{'}-\tau.
\end{eqnarray*}
Thus, $\varphi_{p,q}^{-1}$ meets the desired H\"{o}lder conditions on $\varphi_{p,q}(L^{-1}(\mathbb{R}_{+}))$.

We now complete the proof by showing that $\varphi_{p,q}^{-1}$ meets the desired H\"{o}lder conditions on all of
$\varphi_{p,q}(\Lambda)$. Let $\tau,\tau^{'} \in \varphi_{p,q}(\Lambda)$.
If $\tau,\tau^{'} \in \varphi_{p,q}(L^{-1}(\mathbb{R}_{+})) $, then the H\"{o}lder conditions are met by the the preceding paragraph. On the other hand, suppose that $\tau \in \varphi_{p,q}(L^{-1}(\mathbb{R}_{+})) $
and $\tau^{'} \in \varphi_{p,q}(\Lambda) ~\backslash~ \varphi_{p,q}(L^{-1}(\mathbb{R}_{+})) $. Let $t \in  L^{-1}(\mathbb{R}_{+})$ and $t^{'} \in \Lambda \backslash  L^{-1}(\mathbb{R}_{+})$
be such that $\varphi_{p,q}(t)=\tau$ and $\varphi_{p,q}(t^{'})=\tau^{'}$. By the identity \eqref{Display:Disjoint:Intervals}, which holds by condition \ref{middle:cond:lambda:set}, it follows that $t^{'}=G_i$ for some $i \geq 1$. Thus, using the fact by condition \ref{middle:cond:lambda:set} that there are no isolated points in $\Lambda$ and that the identities \eqref{display:complement:of:lambda} and \eqref{Display:Disjoint:Intervals} hold, it is straightforward to show that there exists a non-decreasing sequence $\{t_k^{'}\}$ of elements of $L^{-1}(\mathbb{R}_{+})$ such that $t_k^{'} \uparrow G_i$ and $L(t_k^{'}) \uparrow L(G_i)$. Hence, by the continuity of $\varphi_{p,q}$ on $\mathbb{R}_{+}$, which is guaranteed by condition \ref{middle:cond:phi:on:r}, it follows that $\tau_k^{'}=\varphi_{p,q}(t^{'}_k) \uparrow \varphi_{p,q}(G_i)=\tau^{'}$. Thus, since $\tau_k^{'} \in \varphi_{p,q}(L^{-1}(\mathbb{R}_{+}))$, and noting that $\varphi_{p,q}^{-1} \circ \varphi_{p,q} (t)=t$ for $t \in \Lambda$ implies that $\varphi_{p,q}^{-1}(\tau_k^{'})=t_k^{'} \uparrow t^{'}=\varphi_{p,q}^{-1}(\tau^{'}) $, it follows that
\begin{eqnarray*}
|\varphi_{p,q}^{-1}(\tau)-\varphi_{p,q}^{-1}(\tau^{'})| ~=~\lim_{k \rightarrow \infty}|\varphi_{p,q}^{-1}(\tau)-\varphi_{p,q}^{-1}(\tau^{'}_k)| &\leq&
\lim_{k \rightarrow \infty}|\tau-\tau^{'}_k|^{1/q}~=~|\tau-\tau^{'}|^{1/q}.
\end{eqnarray*}
In a similar manner, it may be shown that $\varphi_{p,q}^{-1}$ meets the desired H\"{o}lder conditions if $\tau,\tau^{'} \in \varphi_{p,q}(\Lambda) ~\backslash~ \varphi_{p,q}(L^{-1}(\mathbb{R}_{+})) $.
\end{proof}

Now, for each $p > \alpha$ and $1 > q > \alpha/2$, let
\begin{eqnarray}
\Phi_{i} &=& \left(\frac{V_1(\varphi_{p,q} ,[G_i,D_i])}{V_p(Y ,[G_i,D_i])} \right)^{1/p},~~i \geq 1. \label{Display:Definition:Phi:I}
\end{eqnarray}

\begin{lemma}\label{Lemma:Inverse:Of:Varphi:On:Excursions}
For each $p > \alpha$ and $1 > q > \alpha/2$ and $i \geq 1$, it follows that $P^0$-a.s., $Y \circ \varphi_{p,q}^{-1} \circ \varphi_{p,q}(t)=Y(t)$ on $t \in [G_i,D_i]$ . Moreover, $Y \circ \varphi_{p,q}^{-1}$ is $P^0$-a.s. H\"{o}lder continuous on $[\varphi_{p,q}(G_i),\varphi_{p,q}(D_i)]$
 with H\"{o}lder exponent $1/p$ and H\"{o}lder constant $H_i(Y \circ \varphi_{p,q}^{-1}) \leq 1/\Phi_i$.
\end{lemma}

\begin{proof} Let $p > \alpha$ and $1 > q > \alpha/2$  and $i \geq 1$.  Since $\varphi_{p,q}:\mathbb{R}_{+} \mapsto \mathbb{R}_{+}$ is, by Proposition \ref{Proposition:Main:Properties:Of:Varphi:P:Q}, $P^{0}$-a.s. well-defined, it follows by \eqref{Display:Definition:Of:Varphi} that $P^0$-a.s.,
\begin{eqnarray}
\frac{V_1(\varphi_{p,q} ,[G_i,D_i])}{V_p(Y ,[G_i,D_i])} \cdot V_p(Y  ,[G_i,t])&=& V_p\left( \Phi_i Y  ,[G_i,t]\right),~t \in [G_i,D_i],~~\label{Display:Linear:P:Variation:Y:Interval}
\end{eqnarray}
where $\Phi_i$ is as defined in \eqref{Display:Definition:Phi:I} above. Next, define the right-continuous inverse function $V_p^{-1}(\Phi_i Y, \cdot):[0, V_p\left( \Phi_i Y  ,[G_i,D_i]\right)] \mapsto [G_i,D_i]$ by
\begin{eqnarray*}
V_p^{-1}(\Phi_i Y, \tau)&=& \inf\{t \in [G_i,D_i] :  V_p( \Phi_i Y  ,[G_i,t]) > \tau \},~\tau \in [0, V_p\left( \Phi_i Y  ,[G_i,D_i]\right)],
\end{eqnarray*}
where $\inf \emptyset = D_i$. Since $Y$ is $P^0$-a.s. continuous, it follows that $V_p(\Phi_i Y, \cdot)$ is $P^0$-a.s. continuous as well,
and so $P^0$-a.s.,
\begin{eqnarray*}
V_p(\Phi_i Y, V_p^{-1}(\Phi_i Y,\tau  ) )&=&\tau,~\tau \in [0, V_p\left( \Phi_i Y  ,[G_i,D_i]\right)].
\end{eqnarray*}
It then follows that $P^0$-a.s.,
\begin{eqnarray*}
Y \circ V_p^{-1}(\Phi_i Y) \circ  V_p(\Phi_i Y, [G_i,t])&=&Y(t),~t \in [G_i,D_i],
\end{eqnarray*}
and, by the discussion immediately following Proposition \ref{Proposition:Finite:Variation:Holder}, $Y \circ V_p^{-1}(\Phi_i Y, \cdot):[0, V_p\left( \Phi_i Y  ,[G_i,D_i]\right)] \mapsto \mathbb{R}^2$
is $P^0$-a.s. H\"{o}lder continuous with H\"{o}lder exponent $1/p$ and H\"{o}lder constant $H(Y \circ V_p^{-1}(\Phi_i Y, \cdot)) \leq 1/\Phi_i$.

Now recall by \eqref{Display:Definition:Of:Varphi} that $P^0$-a.s.,
\begin{eqnarray}
\varphi_{p,q}(t) &=&
\varphi_{p,q}(G_i)+  V_p( \Phi_i Y  ,[G_i,t]),~t \in  [G_i,D_i]. \label{display:vp:time:changed}
\end{eqnarray}
Moreover, by \eqref{Display:Linear:P:Variation:Y:Interval}, $P^0$-a.s.,
\begin{eqnarray*}
V_p\left( \Phi_i Y  ,[G_i,D_i]\right) &=& V_1\left( \varphi_{p,q}  ,[G_i,D_i]\right)~=~\varphi_{p,q}(D_i)-\varphi_{p,q}(G_i).
\end{eqnarray*}
Hence, in order to complete the proof, it suffices by the preceding paragraph to show that $P^0$-a.s.,
\begin{eqnarray*}
V_p^{-1}(\Phi_i Y, \tau) &=&\varphi_{p,q}^{-1}(\varphi_{p,q}(G_i)+\tau),~\tau \in [0, V_p\left( \Phi_i Y  ,[G_i,D_i]\right)]=[0,\varphi_{p,q}(D_i)-\varphi_{p,q}(G_i) ].
\end{eqnarray*}
However, the fact that by Proposition \ref{Proposition:Main:Properties:Of:Varphi:P:Q}, $\varphi_{p,q}$ is $P^0$-a.s. non-decreasing, implies that $P^0$-a.s.,
\begin{eqnarray*}
\varphi_{p,q}^{-1}(\varphi_{p,q}(G_i)+\tau) &=&\inf\{t \geq 0 : \varphi_{p,q}(t) > \varphi_{p,q}(G_i)+\tau \}\\
&=&\inf\{t \in [G_i,D_i] : \varphi_{p,q}(t) > \varphi_{p,q}(G_i)+\tau \},
\end{eqnarray*}
for $\tau \in [0,\varphi_{p,q}(D_i)-\varphi_{p,q}(G_i)]$.
Then, by \eqref{display:vp:time:changed}, $P^0$-a.s.,
\begin{eqnarray*}
\inf\{t \in [G_i,D_i] : \varphi_{p,q}(t) > \varphi_{p,q}(G_i)+\tau \}&=&\inf\{t \in [G_i,D_i] : V_p( \Phi_i Y  ,[G_i,t]) > \tau \}\\
&=&V_p^{-1}(\Phi_i Y, \tau),
\end{eqnarray*}
for $\tau \in [0,\varphi_{p,q}(D_i)-\varphi_{p,q}(G_i)]$.
\end{proof}

The proof of Proposition \ref{Proposition:Variation:Time:Chanaged:Holder} is as follows.

\begin{proof}[Proof of Proposition \ref{Proposition:Variation:Time:Chanaged:Holder}]Combining Lemmas \ref{Lemma:Inverse:Of:Varphi:On:Zero:Set} and \ref{Lemma:Inverse:Of:Varphi:On:Excursions} above, it follows that for each $p > \alpha$ and $1 > q > \alpha/2$, $P^0$-a.s., $Y \circ \varphi^{-1}_{p,q} \circ \varphi_{p,q}(t)=Y(t)$ for $t \in \mathbb{R}_{+}$.
It is also immediate by Proposition \ref{Proposition:Main:Properties:Of:Varphi:P:Q} that $\varphi_{p,q}$ is $P^0$-a.s. non-decreasing on $\mathbb{R}_{+}$, with $\varphi_{p,q}(\mathbb{R}_{+})=\mathbb{R}_{+}$. Hence, in order to complete the proof, it suffices by Proposition \ref{Proposition:Finite:Variation:Holder} and the discussion after it to show that for each $\alpha < p < 2$, there exists some $\alpha/2 < q < p/2 < 1$ such that $Y \circ \varphi^{-1}_{p,q} $ is $P^0$-a.s. locally H\"{o}lder continuous on $\mathbb{R}_{+}$ with H\"{o}lder exponent $1/p$.

Let $\mathcal{H}_3$ be the set of all $\omega \in C_S$ such that the following six conditions hold,
\begin{enumerate}

\item Identity \eqref{Display:Disjoint:Intervals} holds with $D_i \rightarrow + \infty$ as $i \rightarrow \infty$,
        \label{final:cond:disjoint}

\item $\varphi_{p,q}:\mathbb{R}_{+} \mapsto \mathbb{R}_{+}$ is well-defined, non-decreasing and continuous, with $\varphi_{p,q}(\mathbb{R}_{+})=\mathbb{R}_{+}$,  \label{final:cond:phi:well:defined}

\item $\varphi_{p,q}^{-1} \circ \varphi_{p,q} (t)=t$ for $t \in \Lambda$, and $\varphi_{p,q}^{-1}$  H\"{o}lder continuous on $\varphi_{p,q}(\Lambda)$ with H\"{o}lder exponent $1/q$ and H\"{o}lder constant $H(\varphi_{p,q}^{-1}) \leq 1$,
        \label{final:cond:holder:on:lambda}

\item For each $i \geq 1$, $Y \circ \varphi_{p,q}^{-1} \circ \varphi_{p,q}(t)=Y(t)$ on $t \in [G_i,D_i]$ . Moreover, $Y \circ \varphi_{p,q}^{-1}$  is H\"{o}lder continuous on $[\varphi_{p,q}(G_i),\varphi_{p,q}(D_i)]$
 with H\"{o}lder exponent $1/p$ and H\"{o}lder constant $H_i(Y \circ \varphi_{p,q}^{-1}) \leq 1/\Phi_i$,
  \label{final:cond:holder:on:excursion}

\item $X:\mathbb{R}_{+} \mapsto \mathbb{R}^2$ is locally H\"{o}lder continuous with H\"{o}lder exponent $\eta$ for each $0 \leq \eta < 1/2$, \label{final:cond:brownian:holder}

\item The components of $R^{-1}Y$ are non-decreasing on $[G_i,D_i]$ for each $i \geq 1$.
    \label{final:cond:components:y}

\end{enumerate}
 By the discussion in Section \ref{Section:Excursion:Theory:Background}, it is straightforward to show that condition \ref{final:cond:disjoint} holds $P^0$-a.s. Proposition \ref{Proposition:Main:Properties:Of:Varphi:P:Q}, and Lemmas \ref{Lemma:Y:Positive:Variation} and \ref{Lemma:Inverse:Of:Varphi:On:Zero:Set} imply that conditions \ref{final:cond:phi:well:defined} through \ref{final:cond:holder:on:excursion} hold $P^0$-a.s. By Part \ref{Proposition:Skorohod:Problem:Part:1} of Theorem \ref{Proposition:Skorokhod:Problem} of Section \ref{Section:Identification:of:X:and:Skorokhod:Problem}, $X$ is a standard 2-dimensional Brownian motion under $P^0$, and so condition \ref{final:cond:brownian:holder} holds $P^0$-a.s. by Remark 2.12 of \cite{KaratzasShreve}. Condition \ref{final:cond:components:y} holds $P^0$-a.s. by Part \ref{Proposition:Skorohod:Problem:Part:2} of Theorem \ref{Proposition:Skorokhod:Problem}. Thus, $P^{0}(\mathcal{H}_3)=1$. In order to complete the proof, it therefore suffices to show that for each $\omega \in \mathcal{H}_3$, for each $\alpha < p < 2$, there exists some $\alpha/2 < q < p/2 < 1$ such that $Y \circ \varphi^{-1}_{p,q}(\cdot,\omega) $ is  locally H\"{o}lder continuous on $\mathbb{R}_{+}$ with H\"{o}lder exponent $1/p$. Hence, fix an arbitrary $\hat{\omega} \in \mathcal{H}_3$  for the remainder of the proof.

First note that since by condition \ref{final:cond:disjoint}, $D_i \rightarrow + \infty$ as $i \rightarrow \infty$, and since by condition \ref{final:cond:phi:well:defined}, $\varphi_{p,q}(t) \rightarrow +\infty$ as $t \rightarrow +\infty$, it suffices to show that
$Y \circ \varphi^{-1}_{p,q} $ is H\"{o}lder continuous on $[0,\varphi_{p,q}(D_i)]$ for each $i \geq 1$. Let $i \geq 1$ and $\tau,\tau^{'} \in [0,\varphi_{p,q}(D_i)]$ with $0 \leq \tau \leq \tau^{'} \leq \varphi_{p,q}(D_i)$ such that $\tau,\tau^{'} \in \varphi_{p,q}(\Lambda)$. Then, by condition \ref{final:cond:holder:on:lambda}, $\varphi_{p,q}^{-1}(\tau),\varphi_{p,q}^{-1}(\tau^{'}) \in \Lambda $, and so $Y \circ \varphi_{p,q}^{-1}(\tau)=-X \circ \varphi_{p,q}^{-1}(\tau)$ and $Y \circ \varphi_{p,q}^{-1}(\tau^{'})=-X \circ \varphi_{p,q}^{-1}(\tau^{'})$. This then implies that
\begin{eqnarray}
\| Y \circ \varphi_{p,q}^{-1}(\tau^{'})-Y \circ \varphi_{p,q}^{-1}(\tau) \|&=&\| X \circ \varphi_{p,q}^{-1}(\tau^{'})-X \circ \varphi_{p,q}^{-1}(\tau) \|. \label{Display:Increment:Y:On:Image:Lambda}
\end{eqnarray}
Now, by condition \ref{final:cond:brownian:holder}, for each $0 \leq \eta < 1/2$ and $T \geq 0$, there exists a constant $C_{\eta,T}$ such that
\begin{eqnarray}
\| X(t)-X(s)\| &\leq& C_{\eta,T} |t-s|^{\eta},~0 \leq s \leq t \leq T. \label{Display:Brownian:Local:Holder}
\end{eqnarray}
Moreover, by condition \ref{final:cond:holder:on:lambda} and the fact that $\tau,\tau^{'} \in [0,\varphi_{p,q}(D_i)]$ with $\tau,\tau^{'} \in \varphi_{p,q}(\Lambda)$, it follows since $\varphi_{p,q}:\mathbb{R}_{+} \mapsto \mathbb{R}_{+}$ is non-decreasing by condition \ref{final:cond:phi:well:defined}, that $0 \leq \varphi_{p,q}^{-1}(\tau) \leq \varphi_{p,q}^{-1}(\tau^{'}) \leq D_i$. Therefore, by \eqref{Display:Brownian:Local:Holder}, for each $0 \leq \eta < 1/2$, it follows again by condition \ref{final:cond:holder:on:lambda} that
\begin{eqnarray}
\| X \circ \varphi_{p,q}^{-1}(\tau^{'})-X \circ \varphi_{p,q}^{-1}(\tau) \| &\leq&C_{\eta,D_i}|\varphi_{p,q}^{-1}(\tau^{'})- \varphi_{p,q}^{-1}(\tau) |^{\eta} ~\leq~C_{\eta,D_i}|\tau^{'}- \tau |^{\eta/q}. \label{Display:Increment:X:On:Image:Lambda}
\end{eqnarray}
Now let $0 \leq  \eta^{\star} < 1/2$ be such that $ q < \eta^{\star} p$. Note that such an $\eta^{\star}$ exists since by assumption $q < p/2$.
Then $\eta^{\star} /q > 1/p$, and since $0 \leq \tau \leq \tau^{'} \leq \varphi_{p,q}(D_i)$, it follows that
$|\tau^{'}- \tau |^{\eta^{\star}/q} \leq K_{\eta^{\star},i} |\tau^{'}- \tau |^{1/p}$, where $K_{\eta^{\star},i} = (\varphi_{p,q}(D_i))^{\eta^{\star}/q-1/p}$. Hence, by \eqref{Display:Increment:Y:On:Image:Lambda} and \eqref{Display:Increment:X:On:Image:Lambda},
\begin{eqnarray}
\| Y \circ \varphi_{p,q}^{-1}(\tau^{'})-Y \circ \varphi_{p,q}^{-1}(\tau) \| &\leq& K_{\eta^{\star},i} C_{\eta^{\star},D_i}|\tau^{'}-\tau |^{1/p}. \label{Display:Bound:Increment:Y:On:Image:Lambda}
\end{eqnarray}

Next, suppose that $\tau,\tau^{'} \in [0,\varphi_{p,q}(D_i)]$ with $0 \leq \tau \leq \tau^{'} \leq \varphi_{p,q}(D_i)$, and such that $\tau \in \varphi_{p,q}(\Lambda)$ and $\tau^{'} \notin \varphi_{p,q}(\Lambda)$.
In this case, by conditions \ref{final:cond:disjoint} and \ref{final:cond:phi:well:defined}, it follows that $\tau^{'} \in (\varphi_{p,q}(G_j),\varphi_{p,q}(D_j))$ for some $j \geq 1$.  By condition \ref{final:cond:holder:on:excursion}, it follows that
\begin{eqnarray*}
\| Y \circ \varphi_{p,q}^{-1}(\tau^{'})-Y \circ \varphi_{p,q}^{-1}(\varphi_{p,q}(G_j)) \| &\leq (1/\Phi_j)|\tau^{'}-\varphi_{p,q}(G_j) |^{1/p}.
\end{eqnarray*}
Also since $\varphi_{p,q}(G_j) \in \varphi_{p,q}(\Lambda) \cap [0,\varphi_{p,q}(D_i)]$, it follows by \eqref{Display:Bound:Increment:Y:On:Image:Lambda} that
\begin{eqnarray*}
\| Y \circ \varphi_{p,q}^{-1}(\varphi_{p,q}(G_j))-Y \circ \varphi_{p,q}^{-1}(\tau) \| &\leq&K_{\eta^{\star},i} C_{\eta^{\star},D_i}|\varphi_{p,q}(G_j)- \tau |^{1/p}.
\end{eqnarray*}
Hence, since $0 \leq \tau \leq \varphi_{p,q}(G_j) < \tau^{'} \leq \varphi_{p,q}(D_i)$, it follows by the triangle inequality that
\begin{eqnarray}
\| Y \circ \varphi_{p,q}^{-1}(\tau^{'})-Y \circ \varphi_{p,q}^{-1}(\tau) \| &\leq ((1/\Phi_j)+ K_{\eta^{\star},i} C_{\eta^{\star},D_i})|\tau^{'}-\tau |^{1/p}. \label{Display:Bound:Increment:Y:On:Lambda:Not:Lambda}
\end{eqnarray}
Similar reasoning leads to the same inequality if we assume that $\tau,\tau^{'} \in [0,\varphi_{p,q}(D_i)]$ with $0 \leq \tau \leq \tau^{'} \leq \varphi_{p,q}(D_i)$, and such that $\tau \notin \varphi_{p,q}(\Lambda)$ and $\tau^{'} \in \varphi_{p,q}(\Lambda)$.

Next consider $\tau,\tau^{'} \in [0,\varphi_{p,q}(D_i)]$ with $0 \leq \tau \leq \tau^{'} \leq \varphi_{p,q}(D_i)$, and such that $\tau,\tau^{'} \notin \varphi_{p,q}(\Lambda)$. Suppose first that
$\tau \in (\varphi_{p,q}(G_j),\varphi_{p,q}(D_j))$ and $\tau \in (\varphi_{p,q}(G_k),\varphi_{p,q}(D_k))$ with $j \neq k$. In this case, similar reasoning to the above leads to the bound
\begin{eqnarray}
\| Y \circ \varphi_{p,q}^{-1}(\tau^{'})-Y \circ \varphi_{p,q}^{-1}(\tau) \| &\leq ((1/\Phi_j)+ K_{\eta^{\star},i} C_{\eta^{\star},D_i}+(1/\Phi_k))|\tau^{'}-\tau |^{1/p}. \label{Display:Bound:Increment:Y:Not:Lambda:Not:Lambda:Big:Diff}
\end{eqnarray}
On the other hand, if $\tau,\tau^{'} \in (\varphi_{p,q}(G_j),\varphi_{p,q}(D_j))$, it is immediate by condition \ref{final:cond:holder:on:excursion} that
\begin{eqnarray}
\| Y \circ \varphi_{p,q}^{-1}(\tau^{'})-Y \circ \varphi_{p,q}^{-1}(\varphi_{p,q}(\tau)) \| &\leq (1/\Phi_j)|\tau^{'}-\tau |^{1/p}. \label{Display:Bound:Increment:Y:Not:Lambda:Not:Lambda:Small:Diff}
\end{eqnarray}

Combining the cases \eqref{Display:Bound:Increment:Y:On:Image:Lambda},\eqref{Display:Bound:Increment:Y:On:Lambda:Not:Lambda},\eqref{Display:Bound:Increment:Y:Not:Lambda:Not:Lambda:Big:Diff} and \eqref{Display:Bound:Increment:Y:Not:Lambda:Not:Lambda:Small:Diff} we now obtain that
\begin{eqnarray*}
\| Y \circ \varphi_{p,q}^{-1}(\tau^{'})-Y \circ \varphi_{p,q}^{-1}(\tau) \| &\leq (2 \Theta_i+ K_{\eta^{\star},i} C_{\eta^{\star},D_i})|\tau^{'}-\tau |^{1/p}~\textrm{for}~\tau,\tau^{'} \in [0,\varphi_{p,q}(D_i)],
\end{eqnarray*}
where
\begin{eqnarray*}
\Theta_i &=& \sup \{(1/\Phi_k) :  D_k \leq D_i\}.
\end{eqnarray*}
Hence, in order to show that $Y \circ \varphi^{-1}_{p,q} $ is H\"{o}lder continuous on $[0,\varphi_{p,q}(D_i)]$ with H\"{o}lder exponent $1/p$, it suffices to show that $\Theta_i   < + \infty$.
Using the definition of $\Phi_i$ in \eqref{Display:Definition:Phi:I} it suffices to show that there exists  $ 0 < K_i < + \infty$ such that for $k \geq 1$,
\begin{eqnarray}
 1\{D_k \leq D_i\} \left(\frac{V_p(Y ,[G_k,D_k])}{V_1(\varphi_{p,q} ,[G_k,D_k])} \right)^{1/p} &<&   K_i . \label{Display:Inequality:Between:Variations}
\end{eqnarray}

Recall the definition of the reflection matrix $R$ from Section \ref{Subsection:Semimartingale:Decomposition:Of:X:On:Excursion:Intervals}
 and let $\bar{Y}=R^{-1}Y$.  By condition \ref{final:cond:components:y}, the components of $\bar{Y} $ are non-decreasing on $[G_k,D_k]$ for each $k \geq 1$. Thus, since $Z  = X   + Y $ and recalling that $Z(G_k)=Z(D_k)=0$, it follows that
\begin{eqnarray*}
Y(D_k)-Y(G_k) &=& - (X(D_k)-X(G_k)).
\end{eqnarray*}
On the other hand, by the definition of $\bar{Y}$,
\begin{eqnarray*}
\bar{Y}(D_k)-\bar{Y}(G_k) &=&R^{-1}(Y(D_k)-Y(G_k)).
\end{eqnarray*}
Hence, using the local H\"{o}lder continuity \eqref{Display:Brownian:Local:Holder} of $X$ given by condition \ref{final:cond:brownian:holder}, it follows that for $k \geq 1$ such that $D_k \leq D_i$,
\begin{eqnarray}
 \| \bar{Y}(D_k)-\bar{Y}(G_k) \|~\leq~  \| R^{-1}\| \|X(D_k)-X(G_k)\|  &\leq& \| R^{-1}\| C_{\eta,D_i}  (D_i-G_i)^{\eta},\label{Display:YBar:Less:Than:Holder}
\end{eqnarray}
for each $0 \leq  \eta < 1/2$, where $\|R^{-1}\|$ denotes the operator norm of $R^{-1}$. Next, for $x = (x_1,x_2) \in \mathbb{R}^2$, it is straightforward to show that
\begin{eqnarray*}
\|x \|^{\rho}&\leq&|x_1|^{\rho}+|x_2|^{\rho}~\textrm{for}~0\leq \rho \leq 2.
\end{eqnarray*}
Hence, letting $\bar{Y}=(\bar{Y}_1,\bar{Y}_2)$ and using the definition of strong $p$-variation, it is immediate since $\alpha < p < 2$ that
\begin{eqnarray*}
V_p(\bar{Y},[G_k,D_k]) &\leq& V_p(\bar{Y}_1,[G_k,D_k])+V_p(\bar{Y}_2,[G_k,D_k]).
\end{eqnarray*}
However, since $\bar{Y} $ has components which are non-decreasing on $[G_k,D_k]$ and $1 < \alpha < p $, it follows that
\begin{eqnarray*}
V_p(\bar{Y}_{\ell},[G_k,D_k]) &=& (\bar{Y}_{\ell}(D_k)-\bar{Y}_{\ell}(G_k) )^{p}~\leq~ \|\bar{Y}(D_k)-\bar{Y}(G_k)\|^{p}~\textrm{for}~\ell=1,2.
\end{eqnarray*}
Thus, by \eqref{Display:YBar:Less:Than:Holder},
\begin{eqnarray*}
V_p(\bar{Y} ,[G_k,D_k]) &\leq&\bar{C}_{\eta,D_i}(D_k-G_k)^{\eta p},
\end{eqnarray*}
where $\bar{C}_{\eta,D_i}=2 (\| R^{-1}\|C_{\eta,D_i})^{p} $. Next, note that
\begin{eqnarray*}
V_p(Y ,[G_k,D_k])&=&V_p(R\bar{Y} ,[G_k,D_k])~\leq~\|R\|^{p} V_p(\bar{Y},[G_k,D_k]),
\end{eqnarray*}
where $\|R\|$ denotes the operator norm of $R$. Hence,
\begin{eqnarray*}
V_p(Y ,[G_k,D_k])&\leq&\bar{\bar{C}}_{\eta,D_i}(D_k-G_k)^{\eta p}, \end{eqnarray*}
where $\bar{\bar{C}}_{\eta,D_i}=2  (\|R\|\| R^{-1}\|C_{\eta,D_i})^{p} $. Now again let $0 \leq  \eta^{\star} < 1/2$ be such that $ q < \eta^{\star} p$.
Then, since $(D_k-G_k) < D_i$, it follows that
\begin{eqnarray*}
\frac{V_p(Y ,[G_k,D_k])}{(D_k-G_k)^q }&\leq&\bar{\bar{\bar{C}}}_{\eta^{\star},D_i},
\end{eqnarray*}
where $\bar{\bar{\bar{C}}}_{\eta,D_i} =(D_i)^{\eta p-q} \bar{\bar{C}}_{\eta,D_i} $. Hence, \eqref{Display:Inequality:Between:Variations} holds with $K_i=(\bar{\bar{\bar{C}}}_{\eta^{\star},D_i})^{1/p}$.
\end{proof}

\section{Proof of Theorem \ref{Theorem:Converse:Dirichlet:Process:Result}}
\label{section:converse:dirichlet}

Theorem \ref{Theorem:Main:Dirichlet:Process:Result} implies that $Z$ has the unique decomposition $Z=X+Y$,
where for each $z \in S$, $X$ is a standard 2-dimensional Brownian motion started at $z$ and $Y$ is a process of zero energy on
$(C_S,\mathcal{F},\mathcal{F}_t,P^z)$. In the proof
of Theorem \ref{Theorem:Main:Dirichlet:Process:Result}, it was also shown (see \eqref{display:used:before:in:proof} and the ensuing arguments) that for each $2 > p  > \alpha$ and $z \in S$,
\begin{eqnarray}
P^z(V_p(Y,[0,T]) < \infty)&=&1,~T \geq 0. \label{display:use:initally:prove:roughness}
\end{eqnarray}
The fact that \eqref{display:use:initally:prove:roughness} may be extended to all $p > \alpha$ is an immediate consequence of the identity $V_q(Y,[0,T]) \leq (V_p(Y,[0,T]))^{q/p} $ for $1 \leq p \leq q$ (see Remark 2.5 of Chistyakov and Galkin \cite{chistyakov1998maps}). Hence, in order prove Theorem \ref{Theorem:Converse:Dirichlet:Process:Result} it only remains to prove the partial converse result \eqref{Display:Condition:Converse:Main:Dirichlet}.

Recall from Section \ref{Section:Excursion:Theory:Background} the definition of $L$ as the local time of $Z$ at the origin and let $Y \circ L^{-1}$ be the process defined by $Y \circ L^{-1}=\{Y \circ L^{-1}(t), t \geq 0 \}$. For each $t \geq 0$, we have by \eqref{def:inverse:local:time} that $L^{-1}(t)$ is a stopping time relative to the filtration $\mathcal{F}_t$ and so we may define the new filtration $\mathcal{F}_{L^{-1}}=\{\mathcal{F}_{L^{-1}(t)}, t \geq 0 \} $.

\begin{proposition}\label{Lemma:Y:Composed:With:L:Inverse:Is:Levy} For each $z \in S$, the process $Y \circ L^{-1}$ is a L\'{e}vy process on $(C_S,\mathcal{F},\mathcal{F}_{L^{-1}},P^z)$. In particular, it is an  $\alpha$-stable process on  $(C_S,\mathcal{F},\mathcal{F}_{L^{-1}},P^0)$,  i.e. under $P^0$,
\begin{equation}\lambda^{-{1\over\alpha}}  Y \circ L^{-1}(\lambda\cdot)  \buildrel d \over = Y \circ L^{-1}(\cdot),\quad \lambda>0.\label{stable}\end{equation}
\end{proposition}

\begin{proof}
We start with a remark on notation. It is customary to write  $L(t)$ instead of $L(t,\omega)=L(t,Z(\cdot,\omega))$, as we did so far in this paper. However, in this proof occasionally we need to write down the local time at time $t$ calculated for an element of $C_s$ other than $\omega=Z(\cdot,\omega)$. For example, for the additivity property, we need to write down the local time at time $t-s$ applied to $Z(s+\cdot,\omega)$ for some $0\le s\le t$. In order to strike a compromise between clarity and brevity, we shall write $L(t)$ instead of $L(t,\omega)= L(t,Z(\cdot,\omega))$. However,  for $u,s\ge 0$ we shall write $L(u,Z(s+\cdot))$ instead of  $L(u, Z(s+\cdot,\omega))$.
In this spirit, we write the additivity property as
\begin{equation}L(t)=L(s)+ L(t-s,Z(s+\cdot)),\label{add}\end{equation}
$P^z$-a.s., for all $z\in S$, and $0\le s \le t$. The exceptional set on which \eqref{add} does not hold does not depend on $s$ and $t$ since $L$ is assumed to be perfect (see Section \ref{Section:Excursion:Theory:Background}).
Let $0\le a < b$ be arbitrary. Then applying \eqref{add} to $s=L^{-1}(a)$, we obtain
\begin{eqnarray}
L^{-1}(b) &=& \inf\Big\{t\ge L^{-1}(a): L(t)>b\Big\} \nonumber\\
&=&\inf\Big\{t\ge L^{-1}(a): a+ L  \big(t-L^{-1}(a)    ,Z(L^{-1}(a)+\cdot) \big)  >b\Big\}  \nonumber\\
&=&L^{-1}(a)+\inf\Big\{u\ge 0: L  \big(u   ,Z(L^{-1}(a)+\cdot) \big)  >b-a\Big\} \nonumber \\
&=&L^{-1}(a) +L^{-1}(b-a, Z(L^{-1}(a)+\cdot)),\label{localtime}
\end{eqnarray}
$P^z$-a.s., for every $z\in S$.
 Then, for any $B\in{\mathcal B}(S)$, by \eqref{gamma:z}, \eqref{F:additive}, \eqref{localtime}, \eqref{z:is:zero}, the Strong Markov Property (see Theorem 3.14 of Varadhan and Williams \cite{varadhan1985brownian}), and the fact that by \ref {Display:Disjoint:Intervals}  $Z(L^{-1}(a))=0$, we have
\begin{eqnarray*}
&&P^z\left(Y(L^{-1}(b))- Y(L^{-1}(a))\in B\ |\ {\mathcal F}_{L^{-1}(a)}\right)\\
&=&
P^z\left(X(L^{-1}(a))- X(L^{-1}(b))\in B\ |\ {\mathcal F}_{L^{-1}(a)}\right)\\
&=&P^z\left(F_{\gamma^z}\left(Z(\cdot)\right)\left(L^{-1}(a)\right)-  F_{\gamma^z}\left(Z(\cdot)\right)\left(L^{-1}(b)\right)      \in B  \ |\ {\mathcal F}_{L^{-1}(a)}\right)\\
&=&P^z\left(-F_{\gamma^z}\left(Z\left(L^{-1}(a)+\cdot\right)\right)\left(L^{-1}(b) - L^{-1}(a)\right)\in B\ |\ {\mathcal F}_{L^{-1}(a)}\right)\\
&=&P^z\left(-F_{\gamma^z}\left(Z\left(L^{-1}(a)+\cdot\right)\right)\left(L^{-1}\left(b-a,Z\left(L^{-1}(a)+\cdot\right)\right)\right)\in B \ |\ {\mathcal F}_{L^{-1}(a)}\right)\\
&=&P^0\left(-F_{\gamma^z}(Z(\cdot))  \left(L^{-1}(b-a)\right)\in B  \right)\\
&=& P^0\left(-X(L^{-1}(b-a))\in B\right) =P^0\left(Y(L^{-1}(b-a))\in B\right),
\end{eqnarray*}
which shows that $Y \circ L^{-1}$ is indeed a L\'{e}vy process on $(C_S,\mathcal{F},\mathcal{F}_{L^{-1}},P^z)$. Next, we show that it is also an $\alpha$-stable process on $(C_S,\mathcal{F},\mathcal{F}_{L^{-1}},P^0)$.

Recall the function $\phi:S\mapsto{\mathbb R}$ from \cite{williams1987local} given by $\phi(r,\theta)=r^\alpha\cos(\alpha\theta-\theta_1)$, $r\ge 0$, $0\le\theta\le\xi$ ($r$ and $\theta$ are polar coordinates). It is shown in \cite{williams1987local} that $\{\phi(Z(t)), t \ge 0\}$ is a local submartingale on $(C_S, {\mathcal F},\mathcal {F}_t,P^z)$ for each $z\in S$, that can be uniquely decomposed as
\begin{equation}\phi(Z(t))=\phi(Z(0))+M(t)+A(t),\label{decompose}\end{equation}
where $M$ is a continuous local martingale, and $A$ is an adapted, measurable, non-decreasing, continuous process, and $M(0)=A(0)=0$, $P^z$-a.s. (see (2.6) in \cite{williams1987local}). It is also shown in \cite{williams1987local} that the processes $A$ and $M$ can be selected so that they do not depend on $z$.  Moreover, by Theorem 2.6 in \cite{williams1987local} we have that
\begin{equation}L(t)=cA(t),\ t\ge 0,~~ P^z\textrm{-}\hbox{a.s.,}\label{la}\end{equation}
for some positive constant $c$.
Next for all $\lambda>0$ we introduce the transformed processes
$$A_\lambda(t)=\lambda^{-\alpha/2}A(\lambda t),\ \ X_\lambda(t)= \lambda^{-1/2}X(\lambda t),\ \ M_\lambda(t)=\lambda^{-\alpha/2}M(\lambda t),\ \ Z_\lambda(t)=\lambda^{-1/2}Z(\lambda t),~~ t \geq 0.$$
By Lemma 2.1 in  \cite{williams1987local}, we have that $Z_\lambda(\cdot)\buildrel \rm d\over = Z(\cdot)$ under $P^0$. Consider now the following identities given in the proof of Theorem 2.6 of \cite{williams1987local}:
$$\phi( Z(\cdot))\buildrel d\over = \phi( Z_\lambda(\cdot))=\phi(\lambda^{-1/2} Z(\lambda\cdot))=
\lambda^{-{\alpha\over 2}}\phi( Z(\lambda\cdot))=  $$
$$\lambda^{-{\alpha\over 2}}\left( M(\lambda\cdot)+ A(\lambda\cdot)\right) =  M_\lambda(\cdot)+  A_\lambda(\cdot),$$
where the identity in law is understood to be under $P^0$.
By the uniqueness of the decomposition in \eqref{decompose}, $ A_\lambda(\cdot)$ must depend on $ Z_\lambda(\cdot)$ in the same way as  $A(\cdot)$ depends on $ Z(\cdot)$. Since $A(\cdot)=A(\cdot)\circ Z(\cdot)$, so we have
$P^0$-almost surely
$ A_\lambda(\cdot)= A\circ Z_\lambda(\cdot)$. Furthermore, $\{\gamma^0(n),n\ge 1\}\in\Gamma$ implies $\{\lambda^{-1/2}\gamma^0(n),n\ge 1\}\in\Gamma$, thus by Lemma \ref{sub} there exists a sequence $\{\epsilon^\lambda(n), n\ge 1\}\subset \{\gamma^0(n), n\ge 1\}$ such that
\begin{equation*}F_{\lambda^{-1/2}\epsilon^\lambda}(\omega) =W(\cdot,\omega)=X(\cdot,\omega),\ P^0\textrm{-}\hbox{a.s.}
\end{equation*}
By \eqref{linear}, we have that $P^0$-almost surely
$$X_\lambda(t)=\lambda^{-1/2}X(\lambda t) = \lambda^{-1/2} F_{\epsilon^\lambda}(Z(\cdot))(\lambda t) = F_{\lambda^{-1/2}\epsilon^\lambda}(\lambda^{-1/2}Z(\lambda\cdot))( t) =F_{\lambda^{-1/2}\epsilon^\lambda}(Z_\lambda(\cdot))(t).$$
We summarize all these as
$$( X_\lambda(\cdot), A_\lambda(\cdot))= (F_{\lambda^{-1/2}\epsilon^\lambda}\circ Z_\lambda(\cdot),  A(\cdot)\circ Z_\lambda(\cdot)) \buildrel \rm d\over = (F_{\lambda^{-1/2}\epsilon^\lambda}\circ Z(\cdot), A(\cdot)\circ Z(\cdot)) = ( X(\cdot), A(\cdot))$$
under $P^0$, that is,
$$\Bigl(\lambda^{-{1\over 2}}X(\lambda\cdot),\lambda^{-{\alpha\over 2}}A(\lambda\cdot)\Bigr)\buildrel \rm d\over = \Bigl(X(\cdot),A(\cdot)\Bigr).$$
However, the identity \eqref{la} implies
$$\lambda^{-{\alpha\over 2}}L(\lambda t)=c \lambda^{-{\alpha\over 2}}A(\lambda t),$$
and thus we have that
$$\Bigl(\lambda^{-{1\over 2}}X(\lambda\cdot),\lambda^{-{\alpha\over 2}}L(\lambda\cdot)\Bigr)\buildrel d\over = \Bigl(X(\cdot),L(\cdot)\Bigr).$$
We now conclude that
$$\lambda^{-{1\over 2}}X\Bigl(L^{-1}\left(\lambda^{\alpha\over 2}\cdot\right)\Bigr)\buildrel d\over = X\Bigl(L^{-1}(\cdot)\Bigr),$$
which implies \eqref{stable}.
\end{proof}

We now present the proof of Theorem \ref{Theorem:Converse:Dirichlet:Process:Result}.

\begin{proof}[Proof of Theorem \ref{Theorem:Converse:Dirichlet:Process:Result} ]Following the discussion at the outset of the section, it only remains to prove the identity \eqref{Display:Condition:Converse:Main:Dirichlet}.
Let $T \geq 0$. Since $Y=-X$ on $(L^{-1}(\mathbb{R}_{+}))_T^{-} \subset \Lambda_T^{-} \subset [0,T]$, it follows by Definition \ref{Definition:Strong:P:Variation} that
\begin{eqnarray*}
  V_p(Y,[0,T]) &\geq&V_p(Y,(L^{-1}(\mathbb{R}_{+}))_T^{-} )~=~  V_p(X,(L^{-1}(\mathbb{R}_{+}))_T^{-} ),~p \geq 0.
\end{eqnarray*}
Hence it suffices to show that $P^0(V_p(X,(L^{-1}(\mathbb{R}_{+}))_T^{-}) < + \infty )=0$ for $0 < p \leq \alpha$. For each $A > 0$, let $T_A$ be defined by $L^{-1}(A)=T_A$. Since
$L^{-1}$ is $P^{0}$-a.s. right-continuous with  $L^{-1}(0)=0$, it follows that $T_A \Rightarrow 0$ as $A \downarrow 0$. Also recall that $P^{0}(T_A > 0)=1$ for each $A > 0$. Moreover, since $L^{-1}$ is $P^0$-a.s. non-decreasing, it follows by the change-of-variables
for $p$-variation (see (P4) of Chistyakov and Galkin \cite{chistyakov1998maps}) that $V_p(X,(L^{-1}(\mathbb{R}_{+}))_{T_A}^{-})=V_p(X \circ L^{-1}, [0,A]),A > 0$. Thus, in order to complete the proof it suffices to show that for $0 < p \leq \alpha$,
\begin{eqnarray}
 P^0( V_p(X \circ L^{-1},[0,A]) < + \infty) &=&0,~A > 0. \label{to:show}
\end{eqnarray}
Let $\nu$ be the Levy measure of $X\circ L^{-1}$ under $P^0$. According to Bretagnolle \cite{bretagnolle1972p} (see also Simon \cite{simon2004small}), \eqref{to:show} holds if and only if
\begin{equation}\int_{|u|<1}|u|^p\nu(du)=\infty.\label{integral}\end{equation}
However, since by Proposition \ref{Lemma:Y:Composed:With:L:Inverse:Is:Levy}, $Y \circ L^{-1} = -X \circ L^{-1} $ is an $\alpha$-stable process on
$(C_S,\mathcal{F},\mathcal{F}_{L^{-1}},P^0)$, it follows
by Theorem 14.3(ii) of  \cite{sato1999levy} that $\nu(dB)=\int_C \int_0^\infty 1_B(r\xi)r^{-1-\alpha}dr\lambda(d\xi)$, where $C$ is the unit circle in ${\mathbb R}^2$ and $\lambda$ is a finite measure on $C$. \eqref{integral} now follows immediately. \end{proof}

\section{Proof of Theorem \ref{necessary}}
\label{section:esp}

Before providing the proof of Theorem \ref{necessary}, we first proceed with some propositions, starting with a result that condition \eqref{full:space} is necessary for the existence of a solution to the ESP.

\begin{proposition}\label{first:proposition}
Let $z\in S$. If the ESP $(S,d(\cdot))$ for the Brownian motion $X$ on $(C_S,\mathcal{F},\mathcal{F}_t,P^z)$
 has a solution $P^z$-a.s., then \eqref{full:space} holds.

\end{proposition}

\begin{proof}
Let the polar coordinates of $v_i$ be $(\|v_i\|,\beta_i)$ for $i=1,2$. Then $\beta_1=\theta_1+{\pi/2}$ and $\beta_2=-(\pi/2-\xi+\theta_2)$. Thus, since by assumption $1 < \alpha < 2$, it follows that $\xi<\beta_1<\pi$ and $-(\pi-\xi)<\beta_2<0$. In addition, the condition $\alpha>1$ implies that the angle of $v_1$ and $v_2$ on the side which contains $S$ is $\beta_1-\beta_2>\pi$. Let $C=\overline{\rm co}(V \cup \{\alpha v_1, \alpha \geq 0\} \cup \{\alpha v_2, \alpha \geq 0\})$, which is a closed convex cone, and assume that $C\not={\mathbb R}^2$, which implies $C\cap S=\{0\}$.
 Let $\tilde Z,\tilde Y$ be a solution of the ESP $(S,d(\cdot))$ for $X$.  Let $B$ be an arbitrary open set in $C$ such that $0\notin B$, and let $t>0$ arbitrary. Now we have $P^z(X(t)\in B)>0$, and also $X(t)+\tilde Y(t)=\tilde Z(t)\in S$. However, whenever $X(t)\in B$, then $\tilde Y(t)\in C$ implies $X(t)+ \tilde Y(t)\in C$, which implies $X(t)+\tilde Y(t)=0$. But this is impossible, since a non-zero vector and its negative can not be in a closed convex cone that is neither ${\mathbb R}^2$, nor a single line.
\end{proof}

We now proceed to show that if the condition \eqref{full:space} holds, then $(Z,Y)$ solves the ESP $(S,d(\cdot))$ for $X$. We first provide some prepatory lemmas. For each $i=1,2,$ let
\begin{eqnarray*}
T_{\partial S_i} &=& \inf\{t > 0 : Z(t) \in \partial S_i\setminus\{0\}\}.
\end{eqnarray*}

\begin{lemma}\label{Proposition:Excursions:Hit:Both:Boundaries}
For each $i=1,2,$ $P^0(T_{\partial S_i} > 0)=0$.
\end{lemma}

\begin{proof}
Let $i=1$ or $2$, and
$$T_{\partial S_i}^c=\inf\left\{t\ge 0: {Z(ct)\over\sqrt c}\in\partial S_i\setminus\{0\}\right\},~c > 0.$$
Then, by Lemma 2.1 in \cite{williams1987local},
$$P^0\left(T_{\partial S_i}^c>t\right) = P^0\left(T_{\partial S_i}^1 >t\right),~t \geq 0.$$
On the other hand,
$$T_{\partial S_i}^c=\inf\left\{t\ge 0: {Z(ct)}\in\partial S_i\setminus\{0\}\right\}= \inf\left\{{u\over c}\ge 0: {Z(u)}\in\partial S_i\setminus\{0\}\right\}={1\over c}T_{\partial S_i}^1.$$
Hence,
$$P^0(T_{\partial S_i}^1>ct) = P^0(T_{\partial S_i}^1>t),$$
which implies $P^0(T_{\partial S_i}=0) + P^0(T_{\partial S_i}=\infty)=1$. By the Blumenthal $0$-$1$ law, $P^0(T_{\partial S_i}=0)$ is either 0 or 1. Hence, in order to complete the proof it suffices  that  $P^0(T_{\partial S_i}=\infty)<1$, and this is what we shall do.

By Lemma 4.2 in \cite{williams1985recurrence}, for Lebesgue-almost every $t\in(0,\infty)$,
\begin{equation} P^0(Z(t)\in\partial S)=0.\label{zero}\end{equation}
However, by Lemma 2.1 in \cite{williams1987local}, for every $c>0$
$$P^0(Z(t)\in\partial S)=P^0\left({Z(ct)\over\sqrt c}\in \partial S\right) = P^0\left({Z(ct)}\in \partial S\right), $$
and so \eqref{zero} holds for every $t>0$. Now let $t>0$ be arbitrary and
$$T_{\partial S_i}(t) = \inf\{s\ge t: Z(s)\in\partial S_i\setminus\{0\}\}.$$
Then, by the Markov property,
\begin{equation} P^0(T_{\partial S_i}(t)<\infty)= E^0 \left[P^0(T_{\partial S_i}(t)<\infty|Z(t))\right]=E^0\left[P^{Z(t)}(T_{\partial S_i}<\infty)\right].\label{markov}\end{equation}
Let $T_{\partial S}= \inf\{t\ge 0: Z(t)\in\partial S\}$. By Proposition \ref{williams}, for each $z\in \hbox{int}(S)$  the law of $Z(\cdot\wedge T_{\partial S})$ under $P^z$ coincides with the law of a Brownian motion started at $z$ and stopped at the first hitting time of $\partial S$. It follows that for all $z\in\hbox{int} (S)$ we have $P^z(T_{\partial S}<\infty)=1$ and $P^z(Z(T_{\partial S})\in\partial S_i\setminus\{0\})>0$,
which implies $P^z(T_{\partial S_i}<\infty)>0$. Then it follows from \eqref{zero} that the last expression in \eqref{markov} is strictly positive, and so also is the first expression there, i.e.  $P^0(T_{\partial S_i}(t)<\infty)>0$. Thus, $P^0(T_{\partial S_i}<\infty)>0$.
\end{proof}

\begin{lemma}\label{endpoints:are:nice}
For each $z\in S$ and $P^z$-a.e.  $\omega\in C_s$, the following holds: for the endpoint $D$  of each excursion interval of $Z(\omega)$ away from zero and each $\epsilon>0$, there exist $u_1,u_2,u\in(D,D+\epsilon)$ such that
\begin{equation} Z(u_1)\in\partial S_1\setminus\{0\}, \  Z(u_2)\in\partial S_2\setminus\{0\},\ \hbox{and}\  Z(u)=0.\label{all:the:v}\end{equation}
\end{lemma}

\begin{proof} Recall the sequence of stopping times $\{T_k^n,k\ge 1\}$ from the proof of Proposition \ref{main}, that is the sequence of endpoints of excursions away from zero with length larger than $1/n$. By the Strong Markov Property, Proposition \ref{Proposition:Excursions:Hit:Both:Boundaries}, and the regularity of the vertex (see Lemma 2.2 of \cite{williams1987local}), for every $k,n\ge 1$, we have that $P^z$-almost surely there exist  $u_1,u_2,u\in(T_k^n,T_k^n+\epsilon)$ such that \eqref{all:the:v} holds. Since the set of excursions is countable, and $\cup_{n\ge 1}\{T_k^n,k\ge 1\}$ is exactly the set of endpoints of all excursion intervals, the statement follows.
\end{proof}

The following is an immediate corollary of Lemma \ref{endpoints:are:nice}.

\begin{cor}\label{no:isolated} For each $z\in S$ and $P^z$-a.e. $\omega\in C_s$, the set $\Lambda$ has no isolated point.
\end{cor}

We now present the proof that under \eqref{full:space},  $(Z,Y)$ solves the ESP $(S,d(\cdot))$ for $X$.

\begin{proposition}\label{second:proposition}
If \eqref{full:space} holds, then for each $z \in S$, $(Z,Y)$ $P^z$-a.s.  solves the ESP $(S,d(\cdot))$ for the Brownian motion $X$ on $(C_S,\mathcal{F},\mathcal{F}_t,P^z)$.
\end{proposition}

\begin{proof} Let  $z \in S$. We will verify that $P^z$-a.s.  $(Z,Y)$ satisfies Conditions \ref{Definition:ESP:Item:One} through \ref{Definition:ESP:Item:Four} of Definition \ref{Definition:Extended:Skorohod:Problem} with $\psi=X$ and $(\phi,\eta)=(Z,Y)$. Condition \ref{Definition:ESP:Item:One} is clear
since by the Doob-Meyer type decomposition of Theorem \ref{Theorem:Main:Dirichlet:Process:Result}, we have that $Z=X+Y$. Next, Condition \ref{Definition:ESP:Item:Two} is immediate as well since $Z \in C_S$. Now recall that $X$ is a Brownian motion started at $z$
on $(C_S,\mathcal{F},\mathcal{F}_t,P^z)$ and so $X$ is $P^z$-a.s continuous. Hence $Y=Z-X$ is $P^z$-a.s. continuous as well and so Condition \ref{Definition:ESP:Item:Four} holds since $ 0 \in
 \overline{\mathrm{co}}[d(\phi(t))]$ for $t \geq 0$.

In order to complete the proof, it suffices to show that $P^z$-a.s. Condition \ref{Definition:ESP:Item:Three} holds. That is, it suffices to show that $P^z$-a.s. for each $t \geq 0$ and $s \in [0,t]$,
\begin{eqnarray}
Y(t) - Y(s) &\in& \overline{\mathrm{co}}[\cup_{u \in (s,t]}d(Z(u))]. \label{Display:Condtion:Three:For:Z}
\end{eqnarray}
Let $0\le s <t$ arbitrary. We distinguish between two cases. In the first case, the interval $[s,t]$ lies entirely within a single excursion of $Z$ away from zero, i.e. $(s,t)\cap \Lambda=\emptyset.$ In this case, by Part \ref{Proposition:Skorohod:Problem:Part:2} of  Theorem \ref{Proposition:Skorokhod:Problem}, $Y(v)=RU(v)$ where $U$ is non-decreasing on $[s,t]$, and $U_i$ is increasing on $[s,t]$ only at times $v\in[s,t]$ when $Z(v)\in \partial S_i$ for $i=i,2$, and then \eqref{Display:Condtion:Three:For:Z} follows.
Next, we show that \eqref{Display:Condtion:Three:For:Z} holds in the second case, where $(s,t)\cap\Lambda\not=\emptyset$. In this case, let $v\in  (s,t)\cap\Lambda$. By Corollary \ref{no:isolated}, there exists another point $u$ besides $v$ in  $(s,t)\cap\Lambda$, and, by Condition
 \ref{Condition:3:Submartingale} of Definition \ref{vw}, there exists an excursion interval included in $[u,v]$ or in $[v,u]$, depending on whether $u<v$ or $v<u$. In either case, by Lemma \ref{endpoints:are:nice} this implies the existence of $u_1,u_2 \in (s,t)$ such that $Z(u_i)\in\partial S_i\setminus\{0\}$, $i=1,2$. Then, by \eqref{full:space},
\begin{eqnarray*}
\overline{\mathrm{co}}(\cup_{u\in(s,t]}d(Z(u)))~=~
\overline{\mathrm{co}}\left(V\cup\{a v_1,a\ge 0\}\cup\{a v_2,a\ge 0\}\right)~=~\mathbb{R}^2,
\end{eqnarray*}
and thus \eqref{Display:Condtion:Three:For:Z} is trivially satisfied.
\end{proof}

The proof of Theorem \ref{necessary} is now immediate.

\begin{proof} [Proof of Theorem \ref{necessary}]
The theorem follows by the combination of Propositions \ref{first:proposition} and \ref{second:proposition}.
\end{proof}

\bibliographystyle{plain}

\bibliography{bibfile}


\noindent Peter Lakner \\
Stern School of Business, New York University, 44 West 4th Street, New York, NY 10012\\
E-mail address: plakner@stern.nyu.edu\\
~\\
\noindent Josh Reed \\
Stern School of Business, New York University, 44 West 4th Street, New York, NY 10012\\
E-mail address: jreed@stern.nyu.edu\\
~\\
\noindent Bert Zwart \\
CWI, P.O. Box 94079, 1090 GB Amsterdam, Netherlands\\
E-mail address: Bert.Zwart@cwi.nl\\

\end{document}